\documentclass[]{amsart}
\usepackage[utf8]{inputenc}
\usepackage{graphicx}

\DeclareUnicodeCharacter{2212}{-}

\usepackage{amsmath, amssymb}
\usepackage{amsthm}
\usepackage{enumerate}
\usepackage{mathtools}
\usepackage{tikz-cd}
\usepackage{comment}
\usepackage{color}
\usepackage{hyperref}

\newcommand{\ZZZ}{\operatorname{Z}}

\newcommand{\GL}{\operatorname{GL}}

\newcommand{\NS}{\operatorname{NS}}

\newcommand{\p}{\operatorname{\mathbb{P}}}
\newcommand{\C}{\operatorname{\mathbb{C}}}
\newcommand{\Z}{\operatorname{\mathbb{Z}}}
\newcommand{\h}{\operatorname{\mathbb{H}^{\infty}}}

\newcommand{\s}{\mathcal{S}}

\newcommand{\A}{\operatorname{\mathbb{A}}}

\newcommand{\R}{\operatorname{\mathbb{R}}}
\newcommand{\F}{\operatorname{\mathbb{F}}}
\newcommand{\kk}{{\bf k}}
\newcommand{\out}{\operatorname{Out}}

\newcommand{\ZZ}{\mathcal{Z}}
\newcommand{\id}{\operatorname{id}}

\newcommand{\car}{\operatorname{char}}
\newcommand{\cent}{\operatorname{Cent}}

\newcommand{\norm}{\operatorname{Norm}}

\newcommand{\aut}{\operatorname{Aut}}

\newcommand{\Bir}{\operatorname{Bir}}
\newcommand{\Cr}{\operatorname{Cr}}
\newcommand{\PGL}{\operatorname{PGL}}
\newcommand{\SL}{\operatorname{SL}}

\newcommand{\J}{\mathcal{J}}

\newcommand{\B}{\mathcal{B}}

\newcommand{\arccosh}{\operatorname{arccosh}}

\def\dashmapsto{\mapstochar\dashrightarrow}
\newcommand{\Ax}{\operatorname{Ax}}

\def\dashmapsto{\mapstochar\dashrightarrow}

\usepackage[shortlabels]{enumitem}
\setlist[enumerate]{label=\rm{(\arabic*)}}
\setlist[enumerate,2]{label=\rm({\it\roman*})}
\setlist[itemize]{label=\raisebox{0.25ex}{\tiny$\bullet$}}

\theoremstyle{plain}
\newtheorem{theorem}{Theorem}[section]
\newtheorem{lemma}[theorem]{Lemma}

\newtheorem{corollary}[theorem]{Corollary}

\theoremstyle{definition}
\newtheorem{definition}{Definition}[section]

\newtheorem{example}[theorem]{Example}

\newtheorem{remark}{Remark}[section]

\begin{document}

\author{Christian Urech}


\address{Office 682, Huxley Building\\
	Mathematics Department\\
	Imperial College London\\
	180 Queen's Gate\\
	London SW7 2AZ, UK }
\email{christian.urech@gmail.com}
\thanks{The author gratefully acknowledges support by the Swiss National Science Foundation Grant "Birational Geometry"  PP00P2 128422 /1 as well as by the Geldner-Stiftung, the FAG Basel, the Janggen P\"ohn-Stiftung and the State Secretariat for Education,
	Research and Innovation of Switzerland}

\title{Subgroups of elliptic elements of the Cremona group}
\date{February 2018}

\maketitle
\begin{abstract}
	The Cremona group is the group of birational transformations of the complex projective plane. In this paper we classify its subgroups that consist only of elliptic elements using elementary model theory. This yields in particular a description of the structure of torsion subgroups. As an appliction, we prove the Tits alternative for arbitrary subgroups of the Cremona group, generalizing a result of Cantat. We also describe solvable subgroups of the Cremona group and their derived length, refining results from D\'eserti.
\end{abstract}
\tableofcontents

\section{Introduction and results}
\subsection{Subgroups of elliptic elements}
To a complex projective surface $S$ one can associate its group of birational transformations $\Bir(S)$. This group is particularly rich and interesting when $S$ is rational. In this case it is isomorphic to the {\it Cremona group} $\Cr_2(\C)\coloneqq \Bir(\p^2_{\C})$. 

Groups of birational transformations of surfaces have been studied for more than 150 years and much progress has been achieved in the last decades. One of the key techniques for studying their group theoretical properties has been an action by isometries on an infinite dimensional hyperbolic space $\h(S)$, which we can associate to every projective surface $S$. Recall that there are three types of isometries of hyperbolic spaces:
\begin{itemize}
	\item {\it elliptic isometries}, which are the isometries that fix a point in $\h(S)$,
	\item {\it parabolic isometries}, which are the isometries that do not fix any point in $\h(S)$, but fix exactly one point in the boundary $\partial\h(S)$,
	\item {\it loxodromic isometries}, which are the isometries that do not fix any point in $\h(S)$, but fix exactly two points in $\partial\h(S)$.
\end{itemize}
We call an element $f\in\Bir(S)$ {\it elliptic, parabolic} or {\it loxodromic}, if the isometry of $\h(S)$ induced by $f$ is elliptic, parabolic or loxodromic respectively. This notion is linked to the dynamical behavior of $f$. Let $H$ be an ample divisor on $S$, the {\it degree}  $\deg_H(f)\in\Z_+$ of $f$ with respect to $H$ is defined by
\[
\deg_H(f)=f^*H\cdot H,
\]
where $f^*H$ is the total transform of $H$ under $f$. An element $f\in\Bir(S)$ is elliptic if and only if there exists an ample divisor $H$ on $S$ such that the sequence $\{\deg_H(f^n)\}_{n\in\Z_+}$ is bounded and in this case the boundedness also holds for any other ample divisor (see Theorem \ref{dim2}).  By $\deg(f)$ we denote the degree of $f$ with respect to the class of a line in $\p^2$.

In this paper we consider subgroups consisting only of elliptic elements - a class of subgroups that has not been understood very well so far.

\begin{definition}
	A group $G\subset\Bir(S)$ is a {\it group of elliptic elements} if every element in $G$ is elliptic.
\end{definition}

A particular case of subgroups of elliptic elements are bounded subgroups:

\begin{definition}
	A group $G\subset \Bir(S)$ is {\it bounded} if there exists a constant $K$ such that $\deg(g)\leq K$ for all elements $g\in G$. 
\end{definition}

In \cite{MR2504924}, Blanc showed that every bounded subgroup is contained in a maximal bounded subgroup and gave a full classification of maximal bounded subgroups of $\Cr_2(\C)$ (these are exactly the maximal {\it algebraic subgroups}, see Section \ref{maxalgsubgroups}). However, not all groups of elliptic elements in $\Cr_2(\C)$ need to be bounded, as the following two examples show:

\begin{example}\label{presfibration}
	Let $G\subset\Cr_2(\C)$ be the group of elements of the form $(x, y+p(x))$, where $p(x)\in\C(x)$ is a rational function. Then every element in $G$ is elliptic, but $G$ contains elements of arbitrarily high degrees. 
\end{example}

\begin{example}\label{wright}
	In \cite{MR540400}, Wright constructs examples of torsion subgroups of $\aut(\A^2)$, and hence in particular of $\Cr_2(\C)$, that contain elements of arbitrarily high degree. In fact, he shows that there is a subgroup $G$ of $\Cr_2(\C)$ that is isomorphic to the subgroup of roots of unity in $\C^*$ but that is not bounded. In \cite{MR1860869}, Lamy shows that some of the examples of Wright do not preserve any fibration. 
\end{example}

The group in Example \ref{presfibration} preserves a rational fibration and the group in Example~\ref{wright} is a  torsion group. Our main result shows that every subgroup of elliptic elements of $\Cr_2(\C)$ that is not bounded has one of these properties:

\begin{theorem}\label{main}
	Let $G\subset\Cr_2(\C)$ be a subgroup of elliptic elements. Then one of the following is true:
	\begin{enumerate}
		\item $G$ is a bounded subgroup;
		\item $G$ preserves a rational fibration;
		\item $G$ is a torsion subgroup.
	\end{enumerate}
\end{theorem}

In the case of non-rational surfaces, an analogue to Theorem \ref{main} is not hard to prove. We will show:

\begin{theorem}\label{nonrational}
	Let $S$ be a non-rational complex projective surface and $G\subset\Bir(S)$ a subgroup of elliptic elements. Then $G$ is bounded or $G$ preserves a rational fibration.
\end{theorem}

\subsection{Torsion subgroups}
Example \ref{wright} shows that certain torsion groups can be embedded into $\Cr_2(\C)$ in such a way that they are neither bounded nor preserve any fibration. However, the group-structure of torsion subgroups is not more complicated than the group structure of bounded groups: 

\begin{theorem}\label{torsionthm}
	Let $G\subset\Cr_2(\C)$ be a torsion subgroup. Then $G$ is isomorphic to a bounded subgroup of $\Cr_2(\C)$. Moreover, $G$ is isomorphic to a subgroup of $\GL_{48}(\C)$.
\end{theorem}

The main tool to prove Theorem \ref{torsionthm} will be the compactness theorem of Malcev from model theory (Section \ref{compactnesssection}). 

The following corollary is a direct consequence of Theorem \ref{torsionthm}. It can be seen as an analogue of the Theorem of Jordan and Schur:

\begin{corollary}\label{jordanschur}
	There exists a constant $K$ such that every torsion subgroup of $\Cr_2(\C)$ contains a commutative normal subgroup of index $\leq K$.
\end{corollary}

\begin{remark}
	The constant $K$ in Corollary \ref{jordanschur}  can be bounded explicitely,  for instance with the help of a theorem of Schur (\cite[p.258]{curtis1962representation}).
\end{remark}

\subsection{The Tits alternative}
In a next step we will illustrate how Theorem \ref{main} and Theorem \ref{torsionthm} can be used to prove structure theorems on general subgroups of $\Cr_2(\C)$. Recall the following definition:

\begin{definition}
	\item
	\begin{enumerate}
		\item A group $G$ satisfies the {\it Tits alternative} if every subgroup of $G$ is either virtually solvable or contains a non-abelian free subgroup.
		\item A group $G$ satisfies the {\it Tits alternative for finitely generated subgroups} if every finitely generated subgroup of $G$ either is virtually solvable or contains a non-abelian free subgroup.
	\end{enumerate}
\end{definition}

Tits showed that linear groups over fields of characteristic zero satisfy the Tits alternative and that linear groups over fields of positive characteristic satisfy the Tits alternative for finitely generated subgroups (\cite{MR0286898}). Other well-known examples of groups that satisfy the Tits alternative include mapping class groups of surfaces (\cite{MR745513}), $\out(\F_n)$, the outer automorphism group of the free group of finite rank $n$ (\cite{MR1765705}), or hyperbolic groups in the sense of Gromov (\cite{MR919829}). Lamy showed that the Tits alternative holds for $\aut(\A_{\C}^2)$ using its amalgamated product structure that is given by the Theorem of Jung and van der Kulk  and Bass-Serre theory  (\cite{MR1832900}). In \cite{MR2811600}, Cantat established the Tits alternative for finitely generated subgroups of $\Cr_2(\C)$ as part of a series of profound results about the group structure of the plane Cremona group, which he deduces from the action of $\Cr_2(\C)$ by isometries on $\h(\p^2)$. The main obstacle to generalize the theorem of Cantat to arbitrary subgroups was caused by unbounded groups of elliptic elements that do not preserve any fibration. At this point, Theorem \ref{main} steps in. It turns out that it yields the results needed to generalize the result of Cantat:

\begin{theorem}\label{titscremona}
	Let $S$ be a complex compact K\"ahler surface. Then $\Bir(S)$ satisfies the Tits alternative.
\end{theorem}

\subsection{Solvable subgroups}

In \cite{MR3499516}, D\'eserti gives a description of solvable subgroups of $\Cr_2(\C)$. Theorem \ref{main} and \ref{torsionthm}  refine her results. We will moreover complement them with the following observation:

\begin{theorem}\label{derivedlength}
	The derived length of a solvable subgroup of $\Cr_2(\C)$ is at most $ 8$.
\end{theorem}

For each $n\in\Z_+$ there exists an $N\in\Z_+$ such that every solvable subgroup of $\GL_n(\C)$ has derived length $\leq N$. This result seems to go back to Zassenhaus. Theorem \ref{derivedlength} can thus be seen as an analogue to this result.

\begin{remark}
	In \cite{MR0302781} Newman shows that every solvable subgroup of $\GL_2(\C)$ and hence of $\PGL_2(\C)$ is of derived length $\leq 4$ and that this bound is sharp. It follows that every solvable subgroup of the {\it de Jonqui\`eres subgroup} (see Section \ref{dejonq}), which is isomorphic to $\PGL_2(\C)\ltimes\PGL_2(\C(t))$,  is of derived length $\leq 8$. However, it is not clear whether this bound is sharp. We will see in the proof of Theorem \ref{derivedlength} that solvable subgroups of $\Cr_2(\C)$ that are not subgroups of $\PGL_2(\C)\ltimes\PGL_2(\C(t))$  are of derived length at most $6$.
\end{remark}

In \cite{Furter:2016kl}, Furter and Poloni show that the maximal derived length of a solvable subgroup of $\aut(\A^2_{\C})$ is $5$ (and that this bound is optimal).

\subsection{Acknowledgements}
I express my warmest thanks to my PhD-advisors J\'er\'emy Blanc and Serge Cantat for sharing their beautiful view on the Cremona group with me and for their constant support and helpful guidance. I also thank Michel Brion and Ivan Cheltsov for many helpful comments.

\section{Preliminaries}

\subsection{The plane Cremona group}
 The aim of this chapter is to gather some results we need for our purposes. Most of the times we will refer to other sources for the proofs.

\subsection{The bubble space}\label{bubblesection}
Let $S$ be a smooth projective surface. The {\it bubble space} $\B(S)$ is, roughly speaking, the set of all points that belong to $S$ or are infinitely near to $S$. More precisely, $\B(S)$ can be defined as the set of all triples $(y,S',\pi)$, where $S'$ is a smooth projective surface, $y\in S'$ and $\pi\colon S'\to S$ a birational morphism, modulo the following equivalence relation: A triple $(y, S', \pi)$ is equivalent to $(y', S'', \pi')$ if there exists a birational transformation $\varphi\colon S''\dashrightarrow S'$ that restricts to an isomorphism in a neighborhood of $y$, maps $y$ to $y'$ and satifies $\pi'\circ\varphi=\pi$. A point $p\in\B(S)$ that is equivalent to $(x,S,\id)$ is called a {\it proper point} of $S$. All points in $\B(S)$ that are not proper are called {\it infinitely near}. If there is no ambiguity, we will denote a point in the bubble space by $y$ instead of $(y, S', \pi)$.

Let $f\colon S_1\dashrightarrow S_2$ be a birational transformation. By Zariski's theorem (see \cite[Corollary II.12]{MR1406314}) we can write $f=\pi_2\circ \pi_1^{-1}$, where $\pi_1\colon S_3\to S_1, \pi_2\colon S_3\to S_1$ are finite sequences of blow ups. We may assume that there is no $(-1)$-curve in $S_3$ that is contracted by both, $\pi_1$ and $\pi_2$. The {\it base-points} $\B(f)$ of $f$ are the points in $\B(S)$ that are blown up by $\pi_1$. The proper base-points are sometimes called {\it indeterminacy points} of $f$.

A birational morphism $\pi\colon S\to S'$ induces a bijection $(\pi_1)_\bullet\colon\B(S)\to\B(S')\setminus\B(\pi^{-1})$ by sending a point represented by $(x,S,\varphi)$ to the point represented by $(x,X, \pi\circ\varphi)$. A birational transformation of smooth projective surfaces $f\colon S\dashrightarrow S'$ induces a bijection  $f_\bullet\colon\B(S)\setminus\B(f)\to\B(S')\setminus \B(f^{-1})$ by  $f_\bullet\coloneqq (\pi_2)_\bullet\circ(\pi_1)_\bullet^{-1}$, where $\pi_1\colon S''\to S$, $\pi_2\colon S''\to S'$ is a minimal resolution of $f$. 

\subsection{An infinite dimensional hyperbolic space}
Let $S$ be a smooth complex projective surface and denote by $\NS(S)$ its  N\'eron-Severi group, which is a finitely generated free abelian group. Its rank $\rho(X)$ is called the {\it Picard number}. By the Hodge index theorem, the signature of the intersection form on  $\NS(S)$ has signature $(1,\rho(X)-1)$. 

The pull-back of a birational morphism $\pi\colon S'\to S$ yields an injection of  $\NS(X)$ into $\NS(Y)$. The morphism $\pi\colon Y\to X$ can be written as a finite sequence of blow-ups. Let $e_1,\dots, e_k\subset Y$ be the classes of the irreducible components of the exceptional divisor of $\pi$, i.e.\,the classes contracted by $\pi$. Then we have a decomposition 
\begin{equation}\label{decomp}
\NS(Y)=\NS(X)\oplus\Z e_1\oplus\cdots\oplus\Z e_k,
\end{equation}
which is orthogonal with respect to the intersection form. 

Let $\pi_1\colon S_1\to 1$ and $\pi_2\colon S_2\to S$ be birational morphisms of smooth projective surfaces. We say that $\pi_1$ is {\it above} $\pi_2$ if $\pi_2^{-1}\circ\pi_1$ is a morphism. In other words, $\pi_1$ lies above $\pi_2$ if all the points that are blown up by $\pi_2$ are also blown up by $\pi_1$. For any two birational morphisms $\pi_1\colon S_1\to X$ and $\pi_2\colon S_2\to X$ there exists a birational morphism $\pi_3\colon S_3\to X$ that lies above $\pi_1$ and $\pi_2$. 
Consider the set of all birational morphisms of smooth projective surfaces $\pi\colon Y\to X$. Our remark  shows that the corresponding embeddings of the N\'eron-Severi groups $\pi\colon\NS(S')\to\NS(S)$ form a directed family, so the direct limit
\[
\ZZ(S)\coloneqq\lim_{\pi\colon S'\to S}\NS(S')
\] 
exists. It is called the {\it Picard-Manin space} of $S$. The intersection forms on the groups $\NS(S')$ induce a quadratic form on $\ZZ(X)$ of signature $(1,\infty)$. 

Let $p\in\B(S)$ be a point in the bubble space of $S$ and denote by $e_p$ the divisor class of the exceptional divisor of the blow-up of $p$ in the corresponding N\'eron-Severi group, i.e.\,$e_p$ can be seen as a point in $\ZZ(X)$. From Equation \ref{decomp} one deduces that the Picard-Manin space has the following decomposition
\[
\ZZ(S)=\NS(S)\oplus\bigoplus_{p\in\B(X)}\Z e_p.
\]
Moreover, $e_p\cdot e_p=-1$ and $e_p\cdot e_q=0$ for all $p\neq q$, as well as $e_p\cdot l=0$ for all $l\in\NS(S)$. 

 Let $\ZZZ(S)$ be the completion of the real vector space $\ZZ(S)\otimes\R$ obtained in the following way:
\[
\ZZZ(S)\coloneqq\{v+\sum_{p\in\B(S)}a_pe_p\mid v\in\NS(X)\otimes\R, a_p\in\R, \sum_{p\in\B(S)}a_p^2<\infty\}.
\] 
The intersection form extends continuously to a quadratic form on $\ZZZ(S)$ with signature $(1,\infty)$. Let $e_0\in\ZZZ(S)$ be a vector that corresponds to an ample class. We define $\h(S)$ to be the set of all vectors $v$ in $\ZZZ(S)$ such that $v\cdot v=1$ and $e_0\cdot v>0$. This yields a distance $d$ on $\h(S)$ by 
\[
d(u,v)\coloneqq \arccosh(u\cdot v).
\]
It turns out that $\h(S)$ with the metric $d$ is a complete metric space that is hyperbolic. The {\it boundary} $\partial\h(S)$ of $\h(S)$ consists of the one-dimensional vector subspaces in the {\it light cone}, i.e.\,the cone of isotropic vectors with respect to the intersection form. 

To an isometry $h$ of $\h(S)$ we associate $L(h)\coloneqq\inf \{d(h(p),p)\mid p\in\h(S)\}$. If $L(h)=0$ and the infimum is attained, i.e.\,$h$ has a fixed point in $\h(S)$, then $h$ is called {\it elliptic}. If $L(h)=0$ but the infimum is not attained, we call $h$ {\it parabolic}. It can be shown that a parabolic isometry fixes exactly one point $p$ on the border $\partial \h(S)$. If $L(h)>0$ we call $h$ {\it loxodromic}. In this case the set
\[
\{p\in\h(S)\mid d(h(p),p)=L(h)\}
\]
is a geodesic line in $\h(S)$. It is called the {\it axis} $\Ax(h)$ of $h$ and $L(h)$ is called the {\it translation length}. A loxodromic isometry has exactly two fixed points in $\partial\h$, one of them attractive and the other repulsive (see \cite{MR2811600}).

 Often, we will just write $\h$ and $\partial\h$ instead of $\h(S)$ and $\partial\h(S)$ if it is clear from the context, over which surface we are working.

A birational morphism $f\colon S'\to S$ of smooth projective surfaces induces an isomorphism $f_*\colon \ZZ(S')\to\ZZ(S)$. Let $\NS(S')=\NS(S)\oplus\Z e_{p_1}\oplus\dots\oplus \Z e_{p_n}$, where $p_1,\dots, p_n\in \B(X)$ are the points blown up by $f$ and $e_{p_i}$ is the irreducible component in the exceptional divisor that is contracted to $p_i$. The map $f_*$ is then defined by $f_*(e_p)=e_{f_\bullet(p)}$ for all $p\in\B(S')$, $f_*(e_{p_i})=e_{p_i}$ and $f_*(D)=D$ for all $D\in \NS(S)\subset \NS(S')$ (where the inclusion is given by the pull back of $f$). 
A birational map $f\colon S'\dashrightarrow S$ induces an isomorphism $f_*\colon \ZZ(S')\to\ZZ(S)$, which is defined by $f_*=(\pi_2)_*\circ (\pi_1)_*^{-1}$, where $\pi_1\colon S''\to S', \pi_2\colon S''\to S'$ are birational morphisms such that $f=\pi_2\circ\pi_1^{-1}$.  

Now assume that $f\in\Bir(S)$. Then $f_*$ yields an automorphism of $\ZZ(S)\otimes\R$, which extends to an automorphism of the completion $\ZZZ(S)$ and preserves the intersection form. This automorphism thus preserves the hyperboloid $\h(S)$ and it induces an isometry on~$\h(S)$. This gives an action by isometries of $\Bir(S)$ on~$\h(S)$. 

We refer to \cite{MR833513}, where this construction was developed for the first time and \cite{MR2811600} for details and proofs.

An element $f\in\Bir(S)$ is called  {\it elliptic}, if the corresponding isometry on $\h(S)$ is elliptic, {\it parabolic} if the corresponding isometry is parabolic and {\it loxodromic} if the corresponding isometry is loxodromic. The {\it axis} $\Ax(f)$ of a loxodromic element $f\in\Bir(S)$ is the axis in $\h(S)$ of the  isometry of $\h(S)$ corresponding to~$f$.

\subsection{Degrees and types}\label{degreesandtypes}
The importance of the action of $\Bir(S)$ by isometries on $\h(S)$ is a result of the following  correspondence between the dynamical behavior of a birational transformation $f$ of $S$, in particular its degree, and the type of the induced isometry on $\h(S)$:

\begin{theorem}[Gizatullin; Cantat; Diller and Favre]\label{dim2}
	Let $S$ be a complex smooth projective surface with a fixed polarization $H$ and $f\in\Bir(X)$. Then one of the following is true:
	\begin{enumerate}
		\item $f$ is elliptic, the sequence $\{\deg_H(f^n)\}$ is bounded  and there exists a $k\in\Z_+$ and a birational map $\varphi\colon S\dashrightarrow S'$ to a smooth projective surface $S'$ such that $\varphi f^k\varphi^{-1}$ is contained in $\aut^0(S')$, the neutral component of the automorphism group $\aut(S')$.
		\item[(2a)] $f$ is parabolic and $\deg_H(f^n)\sim cn$ for some positive constant $c$ and $f$ preserves a rational fibration, i.e.\,there exists a smooth projective surface $S'$, a birational map $\varphi\colon S\dashrightarrow S'$, a curve $B$ and a fibration $\pi\colon S'\to B$, such that a general fiber of $\pi$ is rational and such that $\varphi f\varphi^{-1}$ permutes the fibers of $\pi$.
		\item[(2b)] $f$ is parabolic and $\deg_H(f^n)\sim cn^2$ for some positive constant $c$ and $f$ preserves a fibration of genus 1 curves, i.e.\,there exists a smooth projective surface $S'$, a birational map $\varphi\colon S\dashrightarrow S'$, a curve $B$ and a fibration $\pi\colon S'\to B$, such that $\varphi f\varphi^{-1}$ permutes the fibers of $\pi$ and such that $\pi$ is an elliptic fibration.
		\item[(3)] $f$ is loxodromic and $\deg_H(f^n)=c \lambda(f)^n+O(1)$ for some positive constant $c$, where $\lambda(f)$ is the {\it dynamical degree} of $f$. In this case, $f$ does not preserve any fibration.
	\end{enumerate}
\end{theorem}

A first main step towards Theorem \ref{dim2} has been done by Gizatullin in \cite{MR563788} (see as well \cite{MR3480704}), where he classified parabolic automorphisms of surfaces. In \cite{diller2001dynamics}, Diller and Favre proved the main result about the possible degree growths (see also \cite{MR2648673}).

Let us also recall the following result, which we will use later:

\begin{theorem}[{\cite[Theorem 6.6]{MR2811600}}]\label{fixnotlox}
	Let $G\subset\Cr_2(\C)$ be a subgroup. If $G$ does not contain any loxodromic element, then $G$ fixes a point in $\h\cup\partial \h$. 
\end{theorem}

\subsection{The de Jonqui\`eres subgroup}\label{dejonq}
A fibration of a surface $S$ is a rational map $\pi\colon S\dashrightarrow C$, where $C$ is a curve such that the general fibers are one-dimensional. We will identify two fibrations $\pi_1\colon S\dashrightarrow C$ and $\pi_2\colon S\dashrightarrow C'$  with each other if there exists an open dense subset $U\subset S$ that is contained in the domain of $\pi_1$ and $\pi_2$ such that the restrictions of $\pi_1$ and $\pi_2$ to $U$ define the same set of fibers. We say that a group $G\subset\Bir(S)$ {\it preserves} a fibration $\pi$ if $G$ permutes the fibers, i.e.\,there exists a rational $G$-action on $C$ such that $\pi$ is a $G$-equivariant map. 
A {\it rational fibration} of a rational surface $S$ is a rational map $\pi\colon S\dashrightarrow \p^1$ such that the general fiber is rational. Recall that, by a Theorem of Noether and Enriques, there exists a birational map $\varphi\colon C\times\p^1\dashrightarrow X$ such that $\pi\circ\varphi$ is the projection onto the first factor. In other words, up to birational transformations there exists just one rational fibration of $\p^2$. 

\begin{definition}
	The {\it de Jonqui\`eres subgroup} $\J$ of $\Cr_2(\C)$ is the subgroup of elements that preserve the pencil of lines through the point $[0:0:1]\in\p^2$. 
\end{definition}

With respect to affine coordinates $[x:y:1]$ an element of $\J$ is of the form
\[
(x,y)\dashmapsto\left(\frac{ax+b}{cx+d},\frac{\alpha(x)y+\beta(x)}{\gamma(x)y+\delta(x)}\right),
\]
where $\left( \begin{array}{cc}
a & b \\ 
c & d
\end{array} \right) \in\PGL_2(\C)$ and $\left( \begin{array}{cc}
\alpha(x) & \beta(x) \\ 
\gamma(x) & \delta(x)
\end{array} \right) \in\PGL_2(\C(x))$.
This induces an isomorphism
\[
\J\simeq \PGL2(\C)\ltimes\PGL_2(\C(x)).
\]
Every subgroup of $\Cr_2(\C)$ that preserves a rational fibration is conjugate to a subgroup of $\J$.

\subsection{Halphen surfaces}
Consider two smooth cubic curves $C$ and $D$ in $\p^2$. Then $C$ and $D$ intersect in $9$ points $p_1,\dots, p_9$ and there is a pencil of cubic curves passing through these 9 points. By blowing up $p_1,\dots, p_9$, we obtain a rational surface $X$ with a fibration $\pi\colon X\to\p^1$ whose fibers are genus 1 curves. More generally, we can consider a pencil of curves of degree $3m$ for any $m\in\Z_+$ and blow up its base-points to obtain a surface $X$. Such a pencil of genus 1 curves is called a {\it Halphen pencil} and the surface $X$ a {\it Halphen surface of index $m$}. A surface $X$ is Halphen if and only if the linear system $|-mK_X|$ is one-dimensional, has no fixed component and is base-point free.  Up to conjugacy by birational maps, every pencil of genus 1 curves of $\p^2$ is a Halphen pencil and Halphen surfaces are the only examples of rational elliptic surfaces. We refer to \cite{MR2904576} and \cite[Chapter 10]{MR1392959} for proofs and more details. A subgroup $G$ of $\Cr_2(\C)$ that preserves a pencil of genus 1 curves is therefore conjugate to a subgroup of the automorphism group of some Halphen surface by the following lemma:

\begin{lemma}\label{halphenaut}
	Let $X$ be a Halphen surface and $f\in\Bir(X)$ a birational transformation that preserves the Halphen pencil, then $f\in\aut(X)$.
\end{lemma}

\begin{proof}
	Since the Halphen pencil is defined by a multiple of $-K_X$, the class of the anticanonical divisor, every birational transformation of a Halphen surface that preserves the Halphen fibration, preserves $K_X$, the class of the canonical divisor. Assume that $f$ is not an automorphism and let $Z$ be a minimal resolution of indeterminacies of $f$ and $\pi, \eta\colon Z\to X$ such that $f=\pi\circ\eta^{-1}$. We have 
	\[
	\eta^*(K_X)+\sum E_i=K_Z=\pi^*(K_X)+\sum F_i,
	\] 
	where the $E_i$ and $F_i$ are the total pull-backs of the exceptional curves; in particular, $E_i^2=-1$, $F_i^2=-1$ and $E_iE_j=0$, $F_iF_j=0$ for $i\neq j$. Since $f$ preserves $K_X$, we have that $\eta^*(K_X)=\pi(K_X)$ and hence $\sum E_i=\sum F_i$. Note that $\sum E_i$ contains at least one $(-1)$-curve $E_k$. Hence $E_k\cdot (\sum E_i)=-1=E_k\cdot (\sum F_i)$. But this implies that $E_k$ is contained in the support of $\sum F_i$, which contradicts the minimality of the resolution.  
\end{proof}

The automorphism groups of Halphen surfaces are studied in \cite{MR563788} and in \cite{MR2904576}, see also \cite{MR3480704}. We need the following result, which can be found in \cite[Remark 2.11]{MR2904576}:

\begin{theorem}\label{halphenstructure}
	Let $X$ be a Halphen surface. Then there exists a homomorphism $\rho\colon \aut(X)\to\PGL_2(\C)$ with finite image such that $\ker(\rho)$ is an extension of an abelian group of rank $\leq 8$ by a cyclic group of order dividing $24$. 
\end{theorem}

We also recall the following result from \cite{MR2811600}:

\begin{lemma}\label{noloxfibration}
	Let $G\subset\Cr_2(\C)$ be a group that does not contain any loxodromic element but contains a parabolic element. Then $G$ is conjugate to a subgroup of the de Jonqui\`eres group $\J$ or to a subgroup of $\aut(Y)$, where $Y$ is a Halphen surface.
\end{lemma}

\begin{proof}
	By Theorem \ref{fixnotlox}, $G$ fixes a point $q\in\h\cup\partial\h$. Let $f\in G$ be a parabolic element. By definition, $f$ has no fixed point in $\h$ and a unique fixed point $p\in\partial\h$. It follows that $p=q$. By Theorem \ref{dim2}, there exists a birational map $\varphi\colon \p^2\dashrightarrow Y$, a curve $C$ and a fibration $\pi\colon Y\to C$  such that $\varphi f\varphi^{-1}$ permutes the fibers of $\pi$. In particular, $\varphi f\varphi^{-1}$ preserves the divisor class of a fiber $F$ of $\pi$. Being the class of a fiber, $F$ has self-intersection $0$. The point $A\in\ZZZ(\p^2)$ corresponding to $F$ satisfies therefore $A\cdot A=0$ and we obtain that $p\in\partial\h$ corresponds to the line passing through the origin and $A$. It follows that every element in $G$ fixes $A$ and hence that every element in $G$ preserves the divisor class of $F$, i.e.\,every element in $\varphi G\varphi^{-1}$ permutes the fibers of the fibration $\pi\colon Y\to C$. If the fibration is rational, $G$ is conjugate to a subgroup of $\J$. If it is a fibration of genus 1 curves, there exists a Halphen surface $Y$ such that $G$  is conjugate to a subgroup of $\Bir(Y)$ and such that $G$ preserves the Halphen fibration. By Lemma \ref{halphenaut}, $G$ is therefore contained in $\aut(Y)$.
\end{proof}

\subsection{The Zariski topology and algebraic subgroups}\label{zariskitop}
Let $S$ be a complex projective variety. Then the group $\Bir(S)$ can be equipped with the {\it Zariski topology}, which we will now briefly recall.
Let $A$ be an algebraic variety and 
\[
f\colon A\times X\dashrightarrow A\times X
\] 
a birational map of the form $(a,x)\dashmapsto (a,f(a,x))$ that induces an isomorphism between open subsets $U$ and $V$ of $A\times X$ such that the projections from $U$ and from $V$ to $A$ are both surjective. In this way we obtain for each $a\in A$  an element of $\Bir(X)$ defined by $x\mapsto \pi_2(f(a,x))$, where $\pi_2$ is the second projection. Such a map $A\to \Bir(X)$ is called a {\it morphism} or {\it family of birational transformations parametrized by $A$}. The {\it Zariski topology} on $\Bir(X)$ is defined as the finest topology such that all morphisms $f\colon A\to \Bir(X)$ for all algebraic varieties $A$ are continuous with respect to the Zariski topology on $A$.

An {\it algebraic subgroup} of $\Bir(X)$ is the image of an algebraic group $G$ by a morphism $G\to \Bir(X)$ that is an injective group homomorphism. Algebraic groups are closed in the Zariski topology and of bounded degree in the case of $\Bir(X)=\Cr_n(\C)$. Conversely, closed subgroups of bounded degree in $\Cr_n(\C)$ are always algebraic subgroups with a unique algebraic group structure that is compatible with the Zariski topology (see \cite{MR3092478}). In \cite{MR3092478}, it is shown moreover, that all algebraic subgroups of $\Cr_n(\C)$ are linear. 

The sets $\Cr_n(\C)_{\leq d}\subset \Cr_n(\C)$ consisting of all birational transformations of degree $\leq d$ are closed with respect to the Zariski topology. Hence the closure of a subgroup of bounded degree in $\Cr_n(\C)$ is an algebraic subgroup. Every algebraic subgroup of $\Cr_2(\C)$ is contained in a maximal algebraic subgroups and the maximal algebraic subgroups have been classified. We will discuss this classification in detail in Section \ref{maxalgsubgroups}. From this classification one deduces in particular, that every bounded subgroup can be {\it regularized}, i.e\,every bounded subgroup is conjugate to a group of automorphisms of some projective surface. This last fact is true in all dimensions, by a Theorem of Weil (\cite{MR0074083}). 

\begin{lemma}\label{fixelliptic}
	Let $G\subset\Cr_2(\C)$ be a group that fixes a point in $\h$. Then the degree of all elements in $G$ is uniformly bounded and there exists a smooth projective variety $X$ and a birational transformation $\varphi\colon \p^2\dashrightarrow X$ such that $\varphi G\varphi^{-1}\subset \aut(X)$. 
\end{lemma}

\begin{proof}
	Let $p\in\h$ be the fixed point of $G$ and denote by $e_0\in\h$ the class of a line in $\p^2$. Let $g\in G$ be arbitrary. Since the action of $G$ on $\h$ is isometric, $d(g(e_0), p)=d(e_0, p)$, in particular, $d(g(e_0), e_0)\leq 2 d(e_0, p)$. This implies
	$\langle g(e_0), e_0\rangle\leq \cosh(2 d(e_0, p))$ for all $g\in G$. As $\langle g(e_0), e_0\rangle=\deg(g)$, the degree of all elements in $G$ is uniformly bounded and thus can be regularized by the above observation.
\end{proof}

\subsection{Tori and monomial maps}
An algebraic torus of rank $n$ is an algebraic subgroup isomorphic to $(\C^*)^n$. The subgroup of diagonal automorphisms $D_2\subset\PGL_{3}(\C)=\aut(\p^2)$ is a torus of rank $2$. All algebraic tori in $\Cr_2(\C)$ are of rank $\leq 2$ and are conjugate in $\Cr_2(\C)$ to a subtorus of $D_2$ (\cite{bialynicki1967remarks}, \cite{MR0284446}).

An integer matrix $A=(a_{ij})\in M_2(\Z)$ determines a rational map $f_A$ of $\p^2$, which we define, with respect to local coordinates $(x,y)$, by
\[
f_A=(x^{a_{11}}y^{a_{12}},x^{a_{21}}y^{a_{22}}).
\] 
We have $f_A\circ f_B=f_{AB}$ for $A,B\in M_2(\Z)$ and $f_A$ is a birational transformation if and only if $A\in\GL_2(\Z)$. This yields an injective homomorphism $\GL_2(\Z)\to\Cr_{2}(\C)$. By abuse of notation, we will identify its image with $\GL_2(\Z)$. The normalizer of $D_2$ in $\Cr_2(\C)$ is the semidirect product 
\[
\norm_{\Cr_2(\C)}(D_2)=\GL_2(\Z)\ltimes D_2.
\]

If $M\in\GL_2(\Z)$ has spectral radius strictly larger than 1, the birational map $f_M$ is loxodromic. This yields examples of loxodromic elements that normalize an infinite elliptic subgroup. The following theorem by Cantat shows that, up to conjugacy, these are the only examples with this property:

\begin{theorem}[{\cite[Appendix]{MR2881312}}]\label{cantatappendix}
	Let $N$ be a subgroup of $\Cr_2(\C)$ containing at least one loxodromic element. Assume that there exists a short exact sequence 
	\[
	1\to A\to N\to B\to 1,
	\]
	where $A$ is infinite and of bounded degree. Then $N$ is conjugate to a subgroup of $\GL_2(\Z)\ltimes D_2$.
\end{theorem} 

In Theorem \ref{ellipticnorm}, we will generalize Theorem \ref{cantatappendix} to the case where $A$ is a group of elliptic elements.

\begin{lemma}\label{densed2}
	Let $m\in\GL_2(\Z)\subset\Cr_2(\C)$ be a loxodromic monomial map and $\Delta_2\subset D_2$ an infinite subgroup that is normalized by $m$. Then $\Delta_2$ is dense in $D_2$ with respect to the Zariski topology.
\end{lemma}

\begin{proof}
	Let $\overline{\Delta}_2^0$ be the neutral component of the Zariski-closure of $\Delta_2$. If $\overline{\Delta}_2^0$ has a dense orbit on $\p^2$ we are done. Otherwise, the generic orbits of $\overline{\Delta}_2^0$ are of dimension $1$. Since $m$ normalizes $\overline{\Delta}_2^0$ , it preserves its orbits. This implies in particular that $m$ preserves a rational fibration, which is a contradiction to $m$ being loxodromic.
\end{proof}

\subsection{Small cancellation}
Small cancellation has been  developed in various contexts. In \cite{cantat2013normal} the authors use it to show that $\Cr_2(\C)$ is not simple. In this section, we will briefly recall the results needed for our purposes.

 Let $\epsilon,B>0$. Two geodesic lines $L$ and $L'$ in $\h$ are called $(\epsilon, B)$-close, if the diameter of the set $S=\{x\in L\mid d(x, L')\leq \epsilon\}$ is at least $B$, i.e.\,there exist two points in $S$ with distance at least $B$.

\begin{definition}
	Let $G\subset\Cr_2(\C)$ be a subgroup. A loxodromic element $g\in G$ is called {\it rigid in $G$} if there exist $\epsilon, B>0$ such that for every element $h\in G$ we have: $h(\Ax(g))$ is $(\epsilon, B)$-close to $\Ax(g)$ if and only if $h(\Ax(g))=\Ax(g)$.
\end{definition}

\begin{definition}
	Let $G\subset\Cr_2(\C)$ be a subgroup. A loxodromic element $g\in G$ is called {\it tight in $G$} if it is rigid and if, for all $h\in G$, $h(\Ax(g))=\Ax(g)$ implies $hgh^{-1}=g$ or $hgh^{-1}=g^{-1}$.
\end{definition}

\begin{theorem}[{\cite[Theorem 2.10]{cantat2013normal}}]\label{smallcanc}
	Let $G\subset \Cr_2(\C)$ be a subgroup and $g\in G$ an element that is tight in $G$. Then every element $h$ in $\left<\left<g \right>\right>\setminus\{\id\}$, the smallest normal subgroup of $G$ containing $g$, satisfies the following alternative: Either $h$ is a conjugate of $g$ or $h$ is a loxodromic element with strictly larger translation length than $g$. In particular, for $n\geq 2$, $g$ is not contained in $\left<\left<g^n \right>\right>$.
\end{theorem}

In \cite{Shepherd-Barron:2013qq}, Shepherd-Barron classifies tight elements in $G=\Cr_2(\C)$ using Theorem \ref{cantatappendix}:

\begin{theorem}[\cite{Shepherd-Barron:2013qq}]\label{shep}
	In $\Cr_2(\C)$ every loxodromic element is rigid. If $g$ is conjugate to a monomial map, then no power of $g$ is tight. In all the other cases, there exists an integer $n$ such that $g^n$ is tight.
\end{theorem}

Note that if $G\subset\Cr_2(\C)$ is a subgroup and $g\in \Cr_2(\C)$ is a rigid element, then $g$ is rigid in $G$ as well. The same is true for tight elements. However, there might be loxodromic elements $g\in G$ such that $g$ is tight in $G$ but not in $\Cr_2(\C)$. From the proof of Theorem \ref{shep} (see \cite[p.18]{Shepherd-Barron:2013qq}) and Lemma \ref{densed2} the following Theorem follows:

\begin{theorem}\label{sheprel}
	Let $G\subset\Cr_2(\C)$ be a subgroup and $g\in G$ a loxodromic element. The following two conditions are equivalent:
	\begin{enumerate}
		\item no power of $g$ is tight in $G$;
		\item there is a subgroup $\Delta_2\subset G$ that is normalized by $g$ and a birational transformation $f\in\Cr_2(\C)$ such that $f\Delta_2f^{-1}\subset D_2$ is a dense subgroup and $fgf^{-1}\in\GL_2(\Z)\ltimes D_2$.
	\end{enumerate}
\end{theorem}

\section{Finitely generated groups of elliptic elements}\label{finitegen}

In \cite{MR2811600}, Cantat studied finitely generated subgroups of elliptic elements. We will need the following result:

\begin{theorem}[{\cite[Proposition 6.14]{MR2811600}}]\label{cantat}
	Let $\Gamma$ be a finitely generated subgroup of elliptic elements. Then either $\Gamma$ is bounded or $\Gamma$ preserves a rational fibration, i.e.\,$\Gamma$ is conjugate to a subgroup of $\J$.
\end{theorem}

Example \ref{wright} shows that the condition that $\Gamma$ is finitely generated in Theorem~\ref{cantat} is necessary. It is an open question, which has been asked in \cite{MR2811600} and \cite{MR2648673}, whether there exist finitely generated groups of elliptic elements that are not bounded. 

\begin{lemma}\label{boundedorfib}
	Let $G\subset\Cr_2(\C)$ be a group of elliptic elements. Then one of the following is true:
	\begin{enumerate}
		\item[(a)] $G$ preserves a fibration and is therefore conjugate to a subgroup of the de Jonqui\`eres group $\J$ or to a subgroup of $\aut(X)$, where $X$ is a Halphen surface.
		\item[(b)] Every finitely generated subgroup of $G$ is bounded.
	\end{enumerate} 
	Moreover, if $G$ fixes a point $p\in\partial\h$ that does not correspond to the class of a rational fibration $\pi\colon \p\dashrightarrow\p$, then we are in case (b).
\end{lemma}

\begin{proof}
	By Theorem \ref{fixnotlox}, $G$ fixes a point $p\in\h\cup \partial\h$. If $p\in\h$, then $G$ is bounded and we are done. If $p\in\partial\h$, then either $p$ corresponds to the class of a general fiber of some fibration $\pi\colon Y\to\p^1$, where $Y$ is a rational surface. In this case, $G$ preserves this fibration and is therefore conjugate to a subgroup of $\J$ (if the fibration is rational) or to a subgroup of $\aut(X)$, where $X$ is a Halphen surface (if the fibration consists of curves of genus 1). Or $p$ does not correspond to the class of a fibration. Let us prove that in this case (b) holds. Let $\Gamma\subset G$ be any finitely generated subgroup. By Theorem \ref{cantat}, $\Gamma$ is either bounded or it preserves a rational fibration. In the first case we are done. In the second, it follows that $\Gamma$ fixes a point $q\in\partial\h$ that corresponds to the class of the rational fibration that is preserved by $\Gamma$. Hence $q\neq p$ and $G$ therefore fixes the geodesic line through $p$ and $q$. In particular, $G$ fixes a point in $\h$ and is therefore bounded, by Lemma~\ref{fixelliptic}.
\end{proof}

The Burnside problem asks whether a finitely generated torsion group is finite. In general it has a negative answer. However, there are some important classes of groups in which the Burnside property holds, including the Cremona group:

\begin{theorem}[Theorem 7.7, \cite{MR2811600}]\label{burnside}
	Every finitely generated torsion subgroup of $\Cr_2(\C)$ is finite.
\end{theorem}

Theorem \ref{cantat} and \ref{burnside} are crucial for the proof of Theorem \ref{torsionthm}.

\section{Maximal algebraic subgroups}\label{maxalgsubgroups}
In this section we recall some results about algebraic subgroups of $\Cr_2(\C)$. In \cite{MR2504924}, Blanc classified all maximal algebraic subgroups of $\Cr_2(\C)$ (see \cite{MR683251}, \cite{enriques1893sui} for the case of maximal connected algebraic subgroups). There are 11 classes of maximal algebraic subgroups. We summarize them in Theorem~\ref{maxsubgroups}. The notions appearing in Theorem \ref{maxsubgroups} will be recalled in the next sections. 

\begin{theorem}[\cite{MR2504924}]\label{maxsubgroups}
	Every algebraic subgroup of $\Cr_2(\C)$ is contained in a maximal algebraic subgroup. The maximal algebraic subgroups of $\Cr_2(\C)$ are conjugate to one of the following groups:
	\begin{enumerate}
		\item $\aut(\p^2)\simeq \PGL_3(\C)$
		\item $\aut(\p^1\times\p^1)\simeq (\PGL_2(\C))^2\rtimes \Z/2\Z$
		\item $\aut(S_6)\simeq(\C^*)^2\rtimes (\s_3\times \Z/2\Z)$, where $S_6$ is the del Pezzo surface of degree~6.
		\item $\aut(\F_n)\simeq\C[x,y]_n\rtimes\GL_2(\C)/\mu_n$, where $n\geq 2$ and $\F_n$ is the $n$-th Hirzebruch surface and $\mu_n\subset\GL_2(\C)$ is the subgroup of $n$-torsion elements in the center of $\GL_2(\C)$. 
		\item $\aut(S,\pi)$, where $\pi\colon S\to\p^1$ is an exceptional conic bundle.
		\item[(6)-(10)] $\aut(S)$, where $S$ is a del Pezzo surface of degree $5$, $4$, $3$, $2$ or $1$. In this case, $\aut(S)$ is finite.
		\item[(11)] $\aut(S,\pi)$, where $(S,\pi)$ is a $(\Z/2\Z)^2$-conic bundle and $S$ is not a del Pezzo surface. There exists an exact sequence
		\[
		1\to V\to\aut(S,\pi)\to H_V\to 1,
		\]
		where $V\simeq (\Z/2\Z)^2$ and $H_V\subset\PGL_2(\C)$ is a finite subgroup.
	\end{enumerate}
\end{theorem}

In \cite{MR2504924} one finds a more detailed description of the groups above as well as a classification of the conjugacy classes of the maximal algebraic subgroups.

\subsection{Automorphism groups of del Pezzo surfaces}\label{delpezzo}
Recall that a {\it del Pezzo surface} is a smooth projective surface whose anticanonical divisor class is ample. The degree of a del Pezzo surface $S$ is the self-intersection number of its canonical class. It is well-known that the degree of a del Pezzo surface is a positive integer $\leq 9$. A del Pezzo surface is isomorphic to either $\p^2, \p^1\times\p^1$ or to the blow-up $S$ of $r$ general points in $\p^2$. Here, general means that $S$ does not contain a curve of self-intersection $\leq -2$. In this last case, the degree of $S$ is exactly $9-r$. There exists a unique isomorphism class of del Pezzo surfaces of degree $5,6,7$ and $9$, two isomorphism classes of del Pezzo surfaces of degree $8$ and infinitely many isomorphism classes of del Pezzo surfaces of degree $1,2,3$ or $4$ (see for example \cite[Chapter 8]{MR2964027} for proofs and references for these and other results on del Pezzo surfaces). Automorphism groups of del Pezzo surfaces are always algebraic subgroups of $\Cr_2(\C)$ (\cite[Proposition 2.2.6]{MR2504924}) and they are finite if and only if the degree of the corresponding surface is $\leq 5$.

If the degree of a del Pezzo surface $S$ is $5$, then $\aut(S)=\s_5$. A precise descripition of automorphism groups of del Pezzo surfaces of degree $\leq 4$ can be found in the tables in \cite[Section 10]{MR2641179}. This description yields in particular the following:

\begin{theorem}[\cite{MR2641179}]\label{finitesubgroupsorder}
	If the automorphism group of a del Pezzo surface is finite, then it has order at most $648$.
\end{theorem}

\begin{lemma}\label{faithfuldelpezzo}
	If the automorphism group of a del Pezzo surface is finite, then it can be embedded into $\GL_8(\C)$.
\end{lemma}

\begin{proof}
	Let $S$ be a del Pezzo surface such that $\aut(S)$ is finite. This implies that $S$ is of degree $d\leq 5$, hence $S$ is isomorphic to the blow-up of $r=9-d$ general points $p_1,\dots, p_r$  in $\p^2$, where $4\leq r\leq 8$. The N\'eron-Severi space $\NS(S)\otimes\R$ is therefore of dimension $r+1$ and has a basis $[E_0], [E_{p_1}], \dots, [E_{p_r}]$, where $[E_0]$ is the pullback of the class of a line and the $[E_{p_i}]$ are the classes of the exceptional lines $E_{p_i}$ corresponding to the points $p_i$. Since the self-intersection of a class $[E_{p_i}]$ equals $-1$, the exceptional line $E_{p_i}$ is the only representative of $[E_{p_i}]$ on $S$. If $f\in\aut(S)$ acts as the identity on $\NS(S)\otimes\R$, we obtain that $f$ preserves the exceptional lines $E_{p_i}$ for all $i$. Therefore, $f$ induces an automorphism on $\p^2$ that fixes the points $p_i$. But since the $p_i$ are in general position and $r\geq 4$, the automorphism $f$ is the identity. Thus, the action of $\aut(S)$ on $\NS(S)\otimes\R$ is faithful and we obtain a faithful representation $\aut(S)\to\GL_{r+1}(\C)$. As every element $f\in\aut(S)$ fixes the canonical divisor $K_S$, the one-dimensional subspace $\R\cdot K_S$ in $\NS(S)\otimes\R$ is fixed. We project to the orthogonal complement of $K_S$ in $\NS(S)\otimes\R$ and obtain a faithful representation of $\aut(S)$ into $\GL_r(\C)$.
\end{proof}

A del Pezzo surface of degree $6$ is isomorphic to the surface
\[
S_6=\{((x:y:z),(a:b:c))\in\p^2\times\p^2\mid ax=by=cz\},
\]
which is the blow-up of $\p^2$ in three general points.
The group $\aut(S_6)$ is isomorphic to $(\C^*)^2\rtimes (\s_3\times\Z/2\Z)$, where $\s_3$ acts by permuting the coordinates of the two factors simultanously, $\Z/2\Z$ exchanges the two factors and $d\in(\C^*)^2$ acts by sending $((x:y:z),(a:b:c))$ to $(d(x:y:z),d^{-1}(a:b:c))$ ($d(x:y:z)$ is the standard action on $\p^2$). In other words, $\aut(S_6)$ is conjugate to the subgroup $(\s_3\times\Z/2\Z)\ltimes D_2\subset\GL_2(Z)\ltimes D_2$. 

\begin{lemma}\label{delpezzo6}
	The group $\aut(S_6)$ can be embedded into $\GL_6(\C)$.
\end{lemma}

\begin{proof}
	Consider the rational map $f\colon\p^2\dashrightarrow \p^6$ given by 
	\[
	[x:y:z]\dashmapsto [x^2y:x^2z:y^2x:y^2z:z^2x:z^2y:xyz].
	\]
	Then, the rational action of $(\s_3\times\Z/2\Z)\ltimes D_2$ on $f(\p^2)$ extends to a regular action on $\p^6$ that preserves the affine space given by $x_6\neq 0$. This yields an embedding of $(\s_3\times\Z/2\Z)\ltimes D_2$ into $\GL_6(\C)$.
	
\end{proof}

\begin{lemma}\label{delpezzoembedd}
	Let $G\subset\Cr_2(\C)$ be a subgroup that is conjugate to an automorphism group of a del Pezzo surface. Then $G$ can be embedded into $\GL_8(\C)$.
\end{lemma} 

\begin{proof}
	We prove that $\aut(S)$ can be embedded into $\GL_8(\C)$ for all del Pezzo surfaces $S$. 
	If $S$ is a del Pezzo surface of degree $9$ then $S$ is isomorphic to $\p^2$, so $\aut(S)=\PGL_3(\C)\subset\GL_8(\C)$. This corresponds to case $(1)$ of Theorem \ref{maxsubgroups}. If the degree of $S$ is $8$ then $S$ is either isomorphic to $\F_0=\p^1\times\p^1$ or to $\F_1$. The automorphism group $\aut(\F_1)$ is not a maximal algebraic subgroup of $\Cr_2(\C)$ and $\aut(\F_0)=(\PGL_2(\C))^2\rtimes \Z/2\Z\subset \GL_6(\C)$. In the case that $S$ is a del Pezzo surface of degree $7$, $\aut(S)$ is conjugate to a subgroup of $\aut(\p^1\times\p^1)$. A del Pezzo surface of degree $6$ can be embedded into $\GL_6(\C)$, by Lemma \ref{delpezzo6}.
	If the degree of a del Pezzo surface $S$ is $\leq 5$ then $\aut(S)$ is finite and the claim follows from Lemma~\ref{faithfuldelpezzo}.
\end{proof}

\subsection{Automorphism groups of rational fibrations}\label{fibrations}
In this section we consider the cases (4), (5) and (11) of Theorem \ref{maxsubgroups}. 

\subsubsection{Automorphism groups of Hirzebruch surfaces}
In case (4) of Theorem \ref{maxsubgroups}, the semidirect product gives a natural homomorphism $\varphi\colon\aut(\F_n)\to\GL_2(\C)/\mu_n$ whose kernel is isomorphic to $\C^n$. The following lemma shows that all finite subgroups of $\aut(\F_n)$ can be embedded into $\PGL_2(\C)\times\PGL_2(\C)$ for all $n\geq 2$:

\begin{lemma}\label{GLmodmu}
	If $n\geq 2$ is even, the group $\GL_2(\C)/\mu_n$ is isomorphic as an algebraic group to $\PGL_2(\C)\times\C^*$. If $n$ is odd, then $\GL_2(\C)/\mu_n$ is  isomorphic as an algebraic group to $\GL_2(\C)$.
\end{lemma}

\begin{proof}
	Assume that $n\geq 2$ is even, so $n=2k$ for some $k\in\Z_+$. Consider the algebraic group homomorphism $\varphi_e\colon \GL_2(\C)\to\PGL_2(\C)\times\C^*$ given by $A\mapsto ([A], \det(A)^k)$, where $[A]$ denotes the class of $A$ modulo the center of $\GL_2(\C)$. The kernel of $\varphi_e$ consists of the scalar matrices of the form $a\cdot\id$ such that $\det(a)^k=1$; hence, $\ker(\varphi_e)=\mu_n$. Let $(M,c)\in\PGL_2(\C)\times\C^*$. We choose a representative $A\in\SL_2(\C)\subset\GL_2(\C)$ of the class of $M$ and a $n$-th root of $c$ in $\C$, which we denote by $d$. Then $\varphi_e(d\cdot A)=(M,c)$. So $\varphi_e$ is surjective and the claim follows.  
	
	Assume that $n$ is odd; so $n=2k+1$ for  some $k\in\Z_+$. Consider the algebraic group homomorphism $\varphi_o\colon \GL_2(\C)\to\GL_2(\C)$ given by $A\mapsto\det(A)^kA$. Let $B\in\GL_2(\C)$ and $c\in\C$ a $n$-th root of $\det(B)^k$. Then 
	\[
	\varphi_o(c^{-1}\cdot B)=c^{-2k-1}\det(B)^k\cdot B=B,
	\]
	hence $\varphi_o$ is surjective. Moreover, $\ker(\varphi_o)=\mu_n$, and the claim follows.
\end{proof}

\subsubsection{Automorphism groups of exceptional fibrations}
In a next step, we consider the case $\aut(S,\pi)$, where $\pi\colon S\to\p^1$ is an exceptional fibration, i.e.\,case (5) of Theorem \ref{maxsubgroups}. An {\it exceptional fibration} $S$ is by definition a conic bundle with singular fibers above $2n$ points in $\p^1$ and with two sections $s_1$ and $s_2$ of self-intersection $-n$, where $n$ is an integer $\geq 2$ (see \cite{MR2504924}). 

%
%

The proof of the following lemma can be found in \cite[proof of Lemma~4.3.3]{MR2504924}. We briefly recall the arguments of the proof.

\begin{lemma}\label{semidirectexcept}
	Let $\pi\colon S\to \p^1$ be an exceptional fibration. Then $\aut(S,\pi)$ is isomorphic to a subgroup of $\PGL_2(\C)\times\PGL_2(\C)$.
\end{lemma}

\begin{proof}
	There exists a birational morphism $\eta_0\colon S\to\p^1\times\p^1$ of conic bundles that is the blow-up of $2n$ points for some $n\geq 1$; $n$ of them lie on one line $l_1$ of self-intersection $0$ and the other $n$ on another line $l_2$ of self intersection $0$ such that $l_1$ and $l_2$ are disjoint sections of the first projection from $ \p^1\times\p^1$ to $\p^1$ (see \cite[Lemma 4.3.1]{MR2504924}). Let $s_1$ and $s_2$ in $S$ be the strict transforms of $l_1$ and $l_2$ under $\eta$. Hence, $s_1$ and $s_2$ are of self-intersection $-n\leq -2$. The group $\aut(S,\pi)$ then acts on the set $\{s_1,s_2\}$, since $s_1$ and $s_2$ are the unique curves of self-intersection $-n$. This gives an exact sequence
	\[1\to H\to\aut(S,\pi)\to W\to 1,
	\]
	where $W\subset\Z/2\Z$ and $H$ preserves each of the sections $s_1$ and $s_2$. Therefore, $H$ is conjugate by $\eta_0$ to a subgroup of $\aut(\p^1\times\p^1)$ that preserves the structure of a conic bundle, so $H$ is conjugate to a subgroup of $\PGL_2(\C)\times\PGL_2(\C)$. In fact, since $H$ preserves the two lines $l_1$, and $l_2$, it is contained in $\PGL_2(\C)\times \C^*$ and we can write $H=G_1\times A$, where $G_1\subset \PGL_2(\C)$ and  $A\subset\C^*$. We may assume that the $2n$ points blown up by $\eta_0$ are of the form $\{(p_i,[0:1])\}_{i=1}^n$ and $\{(p_i,[1:0])\}_{i=n+1}^{2n}$ for some $p_i\in\p^1$. For each $i=1,\dots, 2n$, let $m_i\in\C[x_0,x_1]$ be a homogeneous form of degree 1 that vanishes on $p_i$. Consider the birational involution of $\p^1\times\p^1$ given by
	\[
	\tau\colon([x_0:x_1], [y_0:y_1])\dashrightarrow \left([x_0:x_1],[y_1\prod_{i=1}^{n}m_i(x_0, x_1): y_1\prod_{i=n+1}^{2n}m_i(x_0, x_1)]  \right).
	\]
	The base-points of $\tau$ are exactly the $2n$ points blown up by $\eta_0$, and $\eta_0^{-1}\tau\eta_0$ is in $\aut(S,\pi)$ and it exchanges the two sections $s_1$ and $s_2$. Hence $W$ is isomorphic to $\Z/2\Z$ and $\tau$ lifts to an automorphism $\tilde{\tau}\coloneqq\eta_0^{-1}\tau\eta_0$. Since $\tau$ fixes the $\p^1$-fibration given by $\pi$, we obtain that $\tilde{\tau}$ commutes with $G_1$. Moreover, $\tilde{\tau}a\tilde{\tau}^{-1}=a^{-1}$ for all $a\in A$ and it follows that $A\rtimes\Z/2\Z$ can be embedded into $\PGL_2(\C)$. In other words, $\aut(S,\pi)$ is of the form $G_1\times(A\rtimes\Z/2\Z)$ and in particular isomorphic to a subgroup of $\PGL_2(\C)\times\PGL_2(\C)$ . 
\end{proof}

\subsubsection{$(\Z/2\Z)^2$-conic bundles} In this section we treat case (11) of Theorem~\ref{maxsubgroups}. A conic bundle $\pi\colon S\to\p^1$ is a {\it $(\Z/2\Z)^2$-conic bundle} if the group $\aut(S/\p^1)$ is isomorphic to $(\Z/2\Z)^2$ and if each of the three involutions of $\aut(S/\p^1)$ fixes an irreducible curve $C$ such that  $\pi\colon C\to\p^1$ is a double covering that is ramified over a positive even number of points. The automorphism group $\aut(S,\pi)$ of a $(\Z/2\Z)^2$-conic bundle is finite and has the structure as described in Theorem \ref{maxsubgroups}. We will need the following Lemma:

\begin{lemma}\label{z2conic}
	Let   $G\subset\Cr_2(\C)$ be an infinite torsion subgroup. Assume that for every finitely generated subgroup $\Gamma\subset G$ there exists a $(\Z/2\Z)^2$-conic bundle $S\to\p^1$ such that $\Gamma$ is conjugate to a subgroup of $\aut(S, \pi)$. Then every finitely generated subgroup of $G$ is isomorphic to a subgroup of $\PGL_2(\C)\times\PGL_2(\C)$.
\end{lemma}

\begin{proof}
	By Theorem \ref{burnside}, every finitely generated subgroup of $G$ is finite. Let $\Gamma\subset G$ be finitely generated and assume that $|\Gamma|>240$. By assumption, $\Gamma$ is conjugate to a subgroup of $\aut(S,\pi)$ for some $(\Z/2\Z)^2$-conic bundle $S\to\p^1$. 
	By Theorem~\ref{maxsubgroups} there is a short exact sequence
		\[
	1\to V\to\aut(S,\pi)\xrightarrow{\varphi} H_V\to 1,
	\]
	where $V=\aut(S/\pi)\simeq (\Z/2\Z)^2$ and $H_V\subset\PGL_2(\C)$ is a finite subgroup. 
	Recall that  finite subgroups of $\PGL_2(\C)$ are isomorphic to $\mathcal{A}_5, \s_4, \mathcal{A}_4, \Z/n\Z$ or $D_{2n}$ for  $n\in\Z_+$ (see \cite{MR2681719}). The condition $|\Gamma|>240$ implies that in our case $H_V$ is dihedral or cyclic. The restriction of $\varphi$ to $\Gamma$ induces an exact sequence
	\[
	1\to V_{\Gamma}\to \Gamma\to H_{V\Gamma}\to 1,
	\]
	where $V_{\Gamma}\subset V$ and $H_{V\Gamma}\subset\PGL_2(\C)$ are finite.
	If $V_{\Gamma}$ is trivial, then $\Gamma\simeq H_{V\Gamma}\subset\PGL_2(\C)$ and we are done. 
	
	Now assume that $V_{\Gamma}\simeq\Z/2\Z$. If $H_{V\Gamma}$ is cyclic, i.e.\,isomorphic to $\Z/l\Z$ for some $l\in\Z_+$, then $\Gamma$ is abelian and isomorphic to $\Z/2\Z\times\Z/l\Z$ or to $\Z/2l\Z$, depending on whether $\Gamma$ contains an element of order $2l$ or not. In both cases we are done. If $H_{V\Gamma}$ is dihedral, i.e.\,isomorphic to $D_{2l}$ for some $l\in\Z_+$, the projection $D_{2l}\to\Z/2\Z$ induces a short exact sequence
	\[
	1\to\Gamma'\to\Gamma\xrightarrow{\alpha}\Z/2\Z,
	\]
	where $\Gamma'$ is either isomorphic to $\Z/2l\Z$ or to $\Z/2\Z\times\Z/l\Z$. Let $g\in\Gamma$ be an element that is not contained in $\Gamma'$. If $\Gamma'\simeq \Z/2l\Z$, then conjugation by $g$ is multiplication by $-1$ and hence $\Gamma\simeq D_{4l}$. If $\Gamma'\simeq\Z/2\Z\times\Z/l\Z$, then $g$ commutes with elements from the first factor and conjugation of elements of the second factor by $g$ corresponds to multiplication by $-1$, hence $\Gamma'\simeq \Z/2\Z\times D_{2l}$. In both cases, $\Gamma$ can be embedded into $\PGL_2(\C)\times\PGL_2(\C)$.  
	
	Finally assume that $V_{\Gamma}$ is isomorphic to $(\Z/2\Z)^2$, i.e.\,$V_{\Gamma}= \aut(S/\pi)$. Let $g\in G$ be an element that is not contained in $\aut(S,\pi)$. Since  $\langle\Gamma, g\rangle \subset G$ is finitely generated, there exists, by assumption, a birational transformation $\varphi\colon S\dashrightarrow S'$ that conjugates $\langle\Gamma, g\rangle \subset G$ to a subgroup of  $\aut(S',\pi')$, where $\pi'\colon S'\to\p^1$ is another $(\Z/2\Z)^2$-conic bundle. In particular, $\varphi$ conjugates $V_{\Gamma}= \aut(S/\pi)$ to a subgroup of $\aut(S',\pi')$. By \cite[Proposition 4.4.6. (5)]{MR2504924}, the birational transformation $\varphi$ is an automorphism and in particular, $G\subset\aut(S,\pi)$ which is a contradiction to $G$ being infinite.
\end{proof}

\section[Proof of Theorem 6.1.3]{Proof of Theorem \ref{main}}
Before we prove Theorem \ref{main} we gather some technical lemmas.

\begin{lemma}\label{twofib}
	Let $g\in\Cr_2(\C)$ be an algebraic element that fixes two different rational fibrations. Then $g$ is of finite order.
\end{lemma}

\begin{proof}
	Assume that the two rational fibrations fixed by $g$ are given by the rational maps $\pi_1, \pi_2\colon \p^2\dashrightarrow \p^1$. We thus obtain a dominant $g$-invariant rational map $\pi\colon \p^2\dashrightarrow \p^1\times\p^1$ given by $\pi=(\pi_1,\pi_2)$. Since $\pi$ is of finite degree, it is locally a finite cover, hence $g$ must be of finite order.
\end{proof}

\begin{lemma}\label{uniquefib}
	Let $G\subset\Cr_2(\C)$ be an algebraic subgroup of dimension $\geq 9$. Then $G$ preserves a unique rational fibration.
\end{lemma}

\begin{proof}
	It follows from Theorem \ref{maxsubgroups} that $G$ is conjugate to a subgroup of $\aut(\mathbb{F}_n)$ for some Hirzebruch surface $\mathbb{F}_n$, $n\geq 2$, and therefore that $G$ preserves a rational fibration $\pi\colon\mathbb{F}_n\to\p^1$. As $G$ permutes the fibers of $\pi$, we obtain a homomorphism $\varphi\colon G\to\PGL_2(\C)$ with a kernel of dimension $\geq 6$. Assume that there exists a second rational fibration $\pi'\colon\F_n\dashrightarrow \p^1$ that is preserved by $G$. We obtain a second homomorphism $\varphi'\colon G\to\PGL_2(\C)$. The restriction of $\pi'$ to $\ker(\pi)$ has a positive dimensional kernel. Hence, the intersection $\ker(\varphi)\cap\ker(\varphi')$ is positive dimensional. In particular, $\ker(\varphi)\cap\ker(\varphi')$ contains an element of infinite order. By Lemma \ref{twofib}, the rational maps $\pi$ and $\pi'$ define the same fibration.
\end{proof}

\begin{lemma}\label{centofc}
	Let $D\subset\Cr_2(\C)$ be an algebraic subgroup that is isomorphic as an algebraic group to $\C^*$. There exists a constant $K(D)$ such that every elliptic element in the centralizer $\cent_{\Cr_2(\C)}(D)$ is of degree $\leq K(D)$.
\end{lemma}

\begin{proof}
	After conjugation by an element $h\in\Cr_2(\C)$, we may assume that $D$  is the group of elements of the form $(cx, y)$, where $c\in\C^*$.  Let $f=(f_1(x,y), f_2(x,y))$ be an elliptic element centralizing $D$, where $f_i(x,y)\in\C(x,y)$. The condition $(f_1(cx,y), f_2(cx,y))=(cf_1(x,y), f_2(x,y))$ implies that we can write $f_1(x,y)=xg_1(y)$  and $f_2(x,y)=g_2(y)$ for some rational functions $g_i(y)\in\C(y)$. As $\{\deg(f^n)\}$ is bounded, we obtain that $g_1(y)$ is constant and as a consequence that $g_2(y)=\frac{a_{11}y+a_{12}}{a_{21}y+a_{22}}$ for some matrix $(a_{ij})\in\PGL_2(\C)$. In particular, $\deg(f)\leq 2$. The constant $K(D)$ thus only depends on the degree of the element $h$.
\end{proof}

\begin{lemma}\label{monomialbounded}
	Let $G$ be a group of monomial elliptic elements, i.e.\,$G\subset\GL_2(\Z)\ltimes(\C^*)^2\subset\Cr_2(\C)$. Then $G$ is bounded.
\end{lemma}

\begin{proof}
	Consider the projection $\pi\colon G\to\GL_2(\Z)$. The kernel of $\pi$ is bounded and all elements in $\pi(G)$ are elliptic. Since all elliptic elements in $\GL_2(\Z)\subset\Cr_2(\C)$ are of finite order, we obtain that $\pi(G)$ is a torsion group. There are only finitely many conjugacy classes of finite subgroups in $\GL_2(\Z)$, hence $\pi(G)$ has to be finite. It follows that $G$ is a finite extension of a bounded subgroup, hence it is bounded.
\end{proof}

\begin{lemma}\label{semisimple}
	Let $G\subset\Cr_2(\C)$ be a group of elliptic elements that normalizes  a semi-simple algebraic subgroup $H\subset \Cr_2(\C)$. Then $G$ is bounded.
\end{lemma}

\begin{proof}
	Since $H$ is semi-simple, its group of inner automorphisms has finite index in its group of algebraic automorphisms. Therefore there exists a constant $N$ such that for any $g\in G$, conjugation by $g^N$ induces an inner automorphism of $H$. Hence there exists an element $h\in H$ such that $g^Nh$ centralizes $H$. Since $H$ is semi-simple, it contains a closed subgroup $D$ isomorphic to an algebraic group to $\C^*$, which is centralized by $g^Nh$. Lemma \ref{centofc} implies that $\deg(g^Nh)$ is bounded by a constant not depending on $g$ or on $N$. Since $H$ is an algebraic subgroup, $\deg(h)$ and thus $\deg(g)$ are also bounded independently of $N$ and $g$ and we obtain that $G$ is bounded. 
\end{proof}

\begin{proof}[Proof of Theorem $\ref{main}$.] 
	Let $G\subset\Cr_2(\C)$ be a group of elliptic elements. We know by Lemma~\ref{boundedorfib}, that either $G$ preserves a rational fibration, or every finitely generated subgroup of $G$ is bounded. In the first case we are done. In the second, we define
	\[
	n\coloneqq\sup\{\dim(\overline{\Gamma})\mid \Gamma\subset G \text{ finitely generated }\}.
	\]
	If $n=0$ the group $G$ is a torsion group and we are done. If $n=\infty$, let $\Gamma\subset G$ be a finitely generated subgroup such that $\overline{\Gamma}$ has dimension $\geq 9$. Lemma \ref{uniquefib} implies that $\Gamma$ preserves a unique rational fibration and this  fibration is, again by Lemma \ref{uniquefib}, preserved as well by $\langle \Gamma, g\rangle$ for all $g\in G$ and we are done. Assume now that $n$ is a positive integer. Let $\Gamma\subset G$ be a finitely generated group such that $\dim(\overline{\Gamma})=n$ and denote by $\overline{\Gamma}^0$ the neutral component of $\Gamma$. For any $g\in G$ the group $\langle \overline{\Gamma}^0, g\overline{\Gamma}^0g^{-1}\rangle$ is connected and contained in the group $\overline{\langle \Gamma, g\Gamma g^{-1}\rangle}$, which is finitely generated and therefore of dimension $\leq n$. This implies $\langle \overline{\Gamma}^0, g\overline{\Gamma}^0g^{-1}\rangle=\overline{\Gamma}^0$ and hence that $\overline{\Gamma}^0$ is normalized by $G$. If $\overline{\Gamma}^0$ is semi-simple we are done by Lemma \ref{semisimple}. Otherwise, the radical $R$ of $\overline{\Gamma}^0$, i.e.\,its maximal connected normal solvable subgroup, is  positive dimensional. Since the radical is unique, it is preserved by automorphisms of $\overline{\Gamma}^0$ and hence in particular normalized by $G$. Let 
	\[
	R=R_k\supset R_{k-1}\supset\dots\supset R_0\supset 1
	\]
	be the derived series of $R$, where $R_0$ is the last non-trivial member. In particular, $R_0$ is positive-dimensional and abelian. This normal series is invariant under automorphisms of $\overline{\Gamma}^0$ and hence invariant under conjugation by elements of $G$. In particular, $G$ normalizes $R_0$. Since $R_0$ is a bounded subgroup it is conjugate to a subgroup of one of the groups from Theorem \ref{maxsubgroups} and in particular, it can be regularized. So we may assume that $G\subset\Bir(S)$ for a smooth projective surface $S$ on which $R_0$ acts regularly. If all the orbits of $R_0$ are of dimension $\leq 1$, we obtain that $G$ preserves a rational fibration and we are done. Now assume that $R_0$ has an open orbit $U$. Since $G$ normalizes $R_0$,  it acts by regular automorphisms on $U$. The action of $R_0$ is faithful, so $R_0$ is of dimension two and therefore isomorphic to $\C^2$, to $\C^*\times\C$ or to $(\C^*)^2$. If $R_0$ is isomorphic to $\C^2$ then $U$ is isomorphic to the affine plane and the action of  $R_0$ on $U$ is given by translations. The normalizer of $\C^2$ in $\aut(\C^2)$ is the group of affine transformations $\GL_2(\C)\ltimes\C^2$, which implies that $G$ is bounded. If $R_0$ is isomorphic to $\C^*\times\C$ we obtain similarly $G\subset\aut(\C^*\times\C)$. The $\C$-fibration of $\C^*\times\C$ is given by the invertible functions, so it is preserved by $\aut(\C^*\times\C)$. In particular, $G$ preserves a rational fibration.  Finally, if $R_0$ is isomorphic to $\C^*\times\C^*$, then $G$ consist of monomial transformations and Lemma~\ref{monomialbounded} implies the claim.
\end{proof}

\begin{proof}[Proof of Theorem \ref{nonrational}]
	Let $S$ be a non-rational  complex projective surface of Kodaira dimension $-\infty$. Then $S$ is ruled, i.e.\,it is birationally equivalent to $\p^1\times C$, where $C$ is a non-rational curve. In this case $\Bir(S)$ preserves the rational fibration given by the projection to $C$. In particular, all subgroups of elliptic elements preserve a rational fibration.
	
	If $S$ is a surface of non-negative Kodaira dimension, then there exists a unique smooth minimal model $S'$ of $S$. In particular we have $\Bir(S)\simeq\Bir(S')=\aut(S')$. The group $\aut(S')$ acts by linear transformations on the cohomology of $S'$ and we obtain a homomorphism $\varphi\colon\aut(S')\to \GL(\NS(S'))$, where $\NS(S')$ is the Neron Severi group of $S'$. The kernel of $\varphi$ is an algebraic subgroup of $\aut(S')$ and hence in particular bounded. Let $G\subset\aut(S')$ be a subgroup of elliptic elements. The restriction of $\varphi$ to $G$ yields an exact sequence
	\[
	1\to G'\to G\to H\to 1,
	\]
	where $G'$ is bounded and $H\subset\GL(\NS(S'))$. By assumption, $G$ consists of elliptic elements, so $H$ is a torsion subgroup.  
	Since finite subgroups in $\GL_n(\Z)$ are bounded, we obtain that $H$ is finite and hence that $G$ is bounded.
\end{proof}

\section{Proof of Theorem \ref{torsionthm}}
In this section we will prove Theorem \ref{torsionthm}. The two main ingredients for our proof is the classification of maximal algebraic subgroups (see Section \ref{maxalgsubgroups}) and the compactness theorem from model theory, which we recall below.

\subsection{The compactness theorem}\label{compactnesssection}
The compactness theorem is a well known result from model theory. It states that a set of first order sentences has a model if and only if each of its finite subsets has a model. The countable version of the theorem has been proven by G\"odel in 1930, the general version is due to Malcev (\cite{mal1940faithful}). We recall the original version as stated by Malcev:

\begin{definition}
	Let $\{x_i\}_{i\in I}$ be a set of variables. A {\it condition} is an expression of the form $F(x_{i_1},\dots, x_{i_k})=0$ or an expression of the form $F_1(x_{i_1},\dots, x_{i_k})\neq 0\vee F_2(x_{i_k},\dots, x_{i_k})\neq0\vee\dots\vee F_l(x_{i_1},\dots, x_{i_k})\neq 0$, where $F$ and the $F_i$ are polynomials with integer coefficients.
	A {\it mixed system} is a set of conditions.
\end{definition}

\begin{definition}
	A mixed system $S$ is {\it compatible} if there exists a field $\kk$ which contains values $\{z_i\}_{i\in I}$ that satisfy $S$.
\end{definition}

\begin{theorem}[Mal'cev, \cite{mal1940faithful}]\label{malcev}
	If every finite subset of a mixed system $S$ is compatible, then $S$ is compatible.
\end{theorem}

Malcev used Theorem \ref{malcev} to deduce that  if for a given group $G$ every finitely generated subgroup can be embedded into $\GL_n(\kk)$ for some field $\kk$ then there exists a field $\kk'$ such that $G$ can be embedded into $\GL_n(\kk')$. Note that a priori nothing can be said about the structure of the field $\kk'$. 

\subsection{Proof of Theorem \ref{torsionthm}}
	If $G$ is finite, the group is bounded, so we assume $G$ to be infinite. 
	
 	First we assume that every finitely generated subgroup of $G$ is isomorphic to a subgroup of $\PGL_3(\C)$. Consider the closed embedding of $\PGL_3(\C)$ into $\GL_8(\C)$ given by the adjoint representation and let $f_1,\dots, f_n$ be  polynomials in the set of variables $\{x_{ij}\}$, where ${1\leq i,j,\leq 8}$, such that the image of $\PGL_3(\C)$ in $\GL_8(\C)$ is the zero set of $f_1,\dots, f_n$.  To every element $g\in G$ we associate a $8\times 8$ matrix of variables $\left(x^g_{ij}\right)$. We now construct a mixed system $S$ consisting of the following conditions:
	
	\begin{enumerate}
		\item  the equations given by the matrix product $(x^f_{ij})(x^g_{ij})=(x^h_{ij})$ for all $f,g,h\in G$ such that $fg=h$;
		\item for all $g\in G\setminus \{\id\}$ the conditions $(\bigvee_{i} x_{ii}^g-1\neq 0)\vee (\bigvee_{i\neq j}x^g_{ij}\neq 0)$;
		\item $x_{ii}^{\id}-1=0$ and $x_{ij}^{\id}=0$ for all $1\leq i,j\leq N$, $i\neq j$;
		\item $f_k(\{x_{ij}\})=0$ for all $k=1,\dots, n$ and all $g\in G$;
		\item $p\neq 0$ for all primes $p\in \Z_+$.
	\end{enumerate}
	First we show that $S$ is compatible. For this, it suffices to show that every finite subset of $S$ is compatible, by Theorem \ref{malcev}. Let $c_1,\dots, c_n\in S$ be finitely many conditions. There are only finitely many variables $x_{ij}^g$ that appear in $c_1,\dots, c_n$. Denote by $\{g_1,\dots, g_l\}\subset G$ the finite set of all elements $g$ in $G$ such that for some $1\leq i,j\leq 8$ the variable $x_{ij}^{g}$ appears in one of the conditions $c_1,\dots, c_n$.
	Consider the finitely generated subgroup $\Gamma=\langle g_1,\dots, g_l\rangle\subset G$. By Theorem~\ref{burnside}, $\Gamma$ is finite and has therefore, by assumption, a faithful representation to $\PGL_3(\C)$. The existence of such a faithful representation implies in particular that $\C$ contains values that satisfy the finite set of conditions $c_1,\dots, c_n$, i.e it is compatible. 
	Hence there exists a field $\kk$ which contains values $z_{ij}^g$ for all $i,j\in\{1,\dots, 8\}$, and all $g\in G$ that satisfy conditions (1) to (5). First we note that $\car(\kk)=0$ because of the conditions (5). Since $G\subset\Cr_2(\C)$, it has at most the cardinality of the continuum, so in particular, the values $\{z_{ij}^g\}$ are contained in a subfield $\kk'$ of $\kk$ with the same cardinality as $\C$, which implies that $\kk'$ can be embedded  into $\C$ as a subfield. So without loss of generality we may assume that $\kk=\C$. Consider now the map $\varphi\colon G\to\PGL_3(\C)$ given by $g\mapsto (z_{ij}^g)_{i,j}$. It is well defined, since the conditions (3) imply $\varphi(\id)=\id$ and by the conditions (1) the image of every element of $G$ is an invertible matrix and the conditions (4) ensure that the images are contained in $\PGL_3(\C)\subset\GL_8(\C)$. The conditions (1) furthermore imply that $\varphi$ is a group automorphism and the conditions (2) ensure that it is injective.
	
	If every finitely generated subgroup of $G$ can be embedded  into $\aut(S_6)\simeq D_2\rtimes (\Z/2\Z\times\s_3)$ we proceed similarly and obtain that $G$ is isomorphic to a subgroup of $\aut(S_6)$. If every finitely generated subgroup of $G$ can be embedded into $\aut(\p^1\times\p^1)\simeq (\PGL_2(\C)\times\PGL_2(\C))\rtimes \Z/2\Z$, we obtain that $G$ is isomorphic to a subgroup of $\aut(\p^1\times\p^1)$. And if for every finitely generated subgroup $\Gamma$ of $G$ there exists a $n>0$ such that $\Gamma$  can be embedded into $\aut(\F_{2n})\simeq \C[x,y]_{2n}\rtimes\GL_2(\C)/\mu_{2n}$ it follows, by using again the same line of argument and Lemma \ref{GLmodmu}, that $G$ is isomorphic to a subgroup of $\GL_2(\C)$ and thus can be embedded into $\PGL_3(\C)$.
	
	Now assume that $G$ contains a finitely generated subgroup $\Gamma_1$ that can not be embedded into $\aut(\p^2)$, a finitely generated subgroup $\Gamma_2$ that can not be embedded into $\aut(S_6)$, a finitely generated subgroup $\Gamma_3$ that can not be embedded into $\aut(\p^1\times\p^1)$ and a finitely generated subgroup $\Gamma_4$ that can not be embedded into $\aut(\F_{2n})$ for all $n>0$. Then the finitely generated subgroup $\Gamma\coloneqq\langle \Gamma_1, \Gamma_2,\Gamma_3, \Gamma_4\rangle$ is not isomorphic to any subgroup of an infinite automorphism group of a del Pezzo surface. After adding finitely many elements, we may assume that the order of $\Gamma$ is $>648$ and hence $\Gamma$ is not isomorphic to any subgroup of an automorphism group of a del Pezzo surface, by Theorem  \ref{finitesubgroupsorder} nor to subgroup of $\aut(\F_{2n})$ for all $n>0$. Let $\Delta\subset G$ be an arbitrary finitely generated subgroup, then the finitely generated subgroup $\langle\Gamma,\Delta\rangle$, and hence in particular $\Delta$, is isomorphic to a subgroup of one of the automorphism groups from case (5) or (11) from Theorem \ref{maxsubgroups} or to a subgroup of $\aut(\F_{2n+1})$ for some $n>0$. Lemma \ref{GLmodmu}, Lemma \ref{semidirectexcept} and Lemma \ref{z2conic} imply that $\Delta$ is isomorphic to a subgroup of $\PGL_2(\C)\times\PGL_2(\C)$. Hence every finitely generated subgroup of $G$ is isomorphic to a subgroup of $\PGL_2(\C)\times\PGL_2(\C)$ and we can use again the compactness theorem to conclude that $G$ is isomorphic to a subgroup of $\PGL_2(\C)\times\PGL_2(\C)$ and therefore in particular of $\aut(\p^1\times\p^1)$.

	It remains to show that every torsion subgroup $G$ is isomorphic to a subgroup of $\GL_{48}(\C)$. 	If $G$ is infinite, it is, by what we have shown above, isomorphic to a subgroup of $\aut(\p^2)$, $\aut(\p^1\times\p^1)$, $\aut(S_6)$, or $\aut(\F_n)$ for some $n\geq 2$. We observe that, by the structure of $\aut(\F_n)$ and Lemma \ref{GLmodmu}, all torsion subgroups of $\aut(\F_n)$ are isomorphic to a subgroup of $\GL_2(\C)$ or of $\PGL_2(\C)\times\C^*$. Since $\PGL_2(\C)$ can be embedded into $\GL_3(\C)$ and $\PGL_3(\C)$ into $\GL_8(\C)$, and $\aut(S_6)$ into $\GL_6(\C)$ by Lemma~\ref{delpezzo6}, we obtain that $G$ is isomorphic to a subgroup of $\GL_8(\C)$. 
	Now assume that $G$ is finite and not contained in an infinite bounded subgroup. If $G$ is contained in the the automorphism group of a del Pezzo surface, it is isomorphic to a subgroup of $\GL_8(\C)$ by Lemma \ref{delpezzoembedd}. If $G$ is contained in the automorphism group of an exceptional fibration, it can be embedded into $\PGL_2(\C)\times\PGL_2(\C)$ by Lemma \ref{semidirectexcept}. Moreover, in \cite[Lemma 6.2.12]{thesisurech} it is shown that if $G$ is contained in the automorphism group of a $(\Z/2\Z)^2$-fibration, then it is isomorphic to a subgroup of $\GL_{48}(\C)$. By Theorem \ref{maxsubgroups}, these are all the possible cases.
	This concludes the proof.
\qed

\section{The Tits alternative}

In this section we prove Theorem \ref{titscremona}. Let $G\subset\Cr_2(\C)$ be a subgroup. We distinguish three cases:
\begin{itemize}
	\item $G$ contains a loxodromic element;
	\item $G$ contains a parabolic element but no loxodromic element;
	\item $G$ is a group of elliptic elements.
\end{itemize}

 The first two cases will be treated similarly as in the proof of the Tits alternative for finitely generated subgroups in \cite{MR2811600}, whereas in the last case the Tits alternative can be deduced from Theorem \ref{main}.

\subsection{The loxodromic case} 
We first start with some preparation. 
The following result is a generalization of Theorem \ref{cantatappendix}:

\begin{theorem}\label{ellipticnorm}
	Let $N$ be a subgroup of $\Cr_2(\C)$ containing at least one loxodromic element. Assume that there exists a short exact sequence 
	\[
	1\to A\to N\to B\to 1,
	\]
	where $A$ is an infinite group of elliptic elements. Then $N$ is conjugate to a subgroup of $\GL_2(\Z)\ltimes D_2$.
\end{theorem}

\begin{proof}
	By Theorem \ref{fixnotlox}, $A$ fixes a point $p\in\partial\h\cup\h$. If $p\in\h$, then $A$ is bounded and we are in the case of Theorem \ref{cantatappendix}. So we may assume that $p\in\partial\h$ and that $p$ is the only fixed point of $A$ in $\partial\h$, since otherwise $A$ fixes the geodesic between $p$ and $q$ and again, $A$ would be bounded. 
	
	Let $f\in N$ be a loxodromic element. Since $f$ normalizes $A$, we obtain that $f$ fixes $p$. Being loxodromic, $f$ does not preserve any fibration and hence, $p$ does not correspond to the class of a fibration. In particular, every finitely generated group of elliptic elements that fixes $p$ is bounded, by Lemma \ref{boundedorfib}. Denote by $G\subset\Cr_2(\C)$ the subgroup of elements that fix $p$. Let $L$ be the one-dimensional subspace of $\ZZZ(\p^2)$ that corresponds to $p$. Since $G$ fixes $p$, its linear action on $\ZZZ(\p^2)$ acts on $L$ by automorphisms preserving the orientation. This yields a group homomorphism $\rho\colon G\to\R_+^*$. The kernel of $\rho$ consists of elliptic elements, since loxodromic elements do not fix any vector in $\ZZZ(\p^2)$, and $G$ does not contain any parabolic element, as $p$ does not correspond to the class of a fibration. Moreover, all elliptic elements in $G$ are contained in $\ker(\rho)$, since $1$ is the only eigenvalue of a transformation of $\ZZZ(\p^2)$ that is induced by an elliptic element (see \cite[Section 4.1.3]{cantatrev}).  Let $f\in G$ be loxodromic; let us show that no power of $f$ is tight. Assume the contrary, i.e.\,that $f^n$ is tight in $G$ for some $n\in\Z$. By Theorem~\ref{smallcanc}, all elements except the identity in the normal subgroup generated by $f^n$ are  loxodromic. In particular, all the elements of the form $gf^ng^{-1}f^{-n}$ are loxodromic, where $g\in G$ is an element that does not commute with $f^n$ (such elements exist since we assumed $N\subset G$ to be infinite and the group $\langle f\rangle$ has finite index in the centralizer of $f$ by \cite[Corollary 2.7]{blanc2016dynamical}). But $\rho(gf^ng^{-1}f^{-n})=\id$ and hence $gf^ng^{-1}f^{-n}$ is elliptic - a contradiction. Hence, no power of $f$ is tight, which implies, by Theorem \ref{sheprel}, that there exists a $h\in\Cr_2(\C)$ and an algebraic subgroup $\Delta_2\subset G$ such that $hfh^{-1}$ is monomial and $h\Delta_2h^{-1}=D_2$. 
	
	Let $\Gamma\subset \ker(\rho)$ be a finitely generated subgroup. Since $\Gamma$ is bounded, the Zariski-closure $\overline{\Gamma}$ of $\Gamma$ is an algebraic subgroup of $G$. Let
	\[
	d\coloneqq\sup\{\dim(\overline{\Gamma})\mid \Gamma\subset \ker(\rho)\text{ finitely generated }\}.
	\]
	First assume that $d$ is finite. Since $\ker(\rho)$ contains a subgroup that is conjugate to $D_2$, we have $d\geq 2$. Let $\Gamma\subset \ker(\rho)$ be a finitely generated subgroup such that $\dim(\overline{\Gamma})=d$ and denote by $\overline{\Gamma}^0$ the neutral component of the algebraic group $\overline{\Gamma}$. Let $f\in G$ be any element. Note that $f\overline{\Gamma}^0f^{-1}$ is again an algebraic subgroup and that $\langle \overline{\Gamma}^0, f\overline{\Gamma}^0f^{-1}\rangle$ is contained in the Zariski-closure of the finitely generated group $\langle\Gamma, f\Gamma f^{-1}\rangle$. By \cite[Chapter 7.5]{MR0396773}, $\langle \overline{\Gamma}^0, f\overline{\Gamma}^0f^{-1}\rangle$ is closed and connected. Since it is of dimension $\leq d$ and contains $\overline{\Gamma}^0$ it equals $\overline{\Gamma}^0$, i.e.\,$f$ normalizes $\overline{\Gamma}^0$. Since $\Gamma\cap\overline{\Gamma}^0$ is infinite, Theorem \ref{cantatappendix} implies that there exists a $g\in\Cr_2(\C)$ such that $gGg^{-1}\subset \GL_2(\Z)\ltimes D_2$ and hence in particular $gNg^{-1}\subset \GL_2(\Z)\ltimes D_2$.
	
	Now assume that $d=\infty$. Let $\Gamma\subset \ker(\rho)$ be a finitely generated subgroup such that $\dim(\overline{\Gamma})\geq 9$. By Lemma \ref{uniquefib}, $\overline{\Gamma}$ preserves a unique rational fibration given by a rational map $\pi\colon\p^2\dashrightarrow \p^1$. Let $g\in \ker(\rho)$ be any element. The algebraic group $\overline{\langle \Gamma, g\rangle}$ preserves again a rational fibration and since it contains $\overline{\Gamma}$, this fibration is given by $\pi$. It follows that $\ker(\rho)$ preserves a rational fibration, which implies that $\ker(\rho)$ is bounded and we can apply Theorem \ref{cantatappendix} to conclude that $gGg^{-1}\subset \GL_2(\Z)\ltimes D_2$ and hence in particular $gNg^{-1}\subset \GL_2(\Z)\ltimes D_2$.
\end{proof}

From Theorem \ref{ellipticnorm} we can in particular draw the following results:

\begin{lemma}\label{nocommonfix}
	Let $f,g\in\Cr_2(\C)$ be two loxodromic elements such that $\Ax(f)\neq\Ax(g)$. Then either $f$ and $g$ do not have a common fixed point on $\partial\h$, or $\langle f,g\rangle $ contains a subgroup $\Delta_2$, and there exists an $h\in\Cr_2(\C)$ such that $h\langle f,g\rangle h^{-1}\subset\GL_2(\Z)\ltimes D_2$ and $h\Delta_2h^{-1}\subset D_2$ is a dense subgroup.
\end{lemma}

\begin{proof}
	Assume that $f$ and $g$ have a common fixed point $p\in\partial\h$. Let $L$ be the one-dimensional subspace of $\ZZZ(\p^2)$ that corresponds to $p$. Since $\langle f,g\rangle$ fixes $p$, its linear action on $\ZZZ(\p^2)$ acts on $L$ by automorphisms preserving the orientation. As in the proof of Theorem \ref{ellipticnorm}, this yields a group homomorphism $\rho\colon \langle f,g\rangle\to\R_+^*$ whose kernel consists of elliptic elements.  Assume that $f^n$ is tight for some $n$. Since $f$ and $g$ have different axes the elements $g$ and $f^n$ do not commute. By Theorem~\ref{smallcanc}, all elements except the identity in the normal subgroup generated by $f^n$ are  loxodromic. In particular, $gf^ng^{-1}f^{-n}$ is loxodromic. But $\rho(gf^ng^{-1}f^{-n})=\id$ and hence $gf^ng^{-1}f^{-n}$ is elliptic - a contradiction. Hence, no power of $f$ is tight in $\langle f,g\rangle $, which implies, by Theorem \ref{sheprel}, that there exists a $h\in\Cr_2(\C)$ and a bounded subgroup $\Delta_2\subset \langle f,g\rangle$ such that $hfh^{-1}$ is monomial and $h\Delta_2h^{-1}$ is a dense subgroup of $D_2$. In particular, $\ker(\rho)$ is infinite as $\Delta_2\subset \ker(\rho)$. The statement now follows from Theorem \ref{ellipticnorm}. 
\end{proof}

\begin{lemma}\label{nocommonfixgroup}
	Let $f,g\in\Cr_2(\C)$ be two loxodromic elements such that $\Ax(f)\neq\Ax(g)$. Then $f$ and $g$ do not have a common fixed point on $\partial\h$.
\end{lemma}

\begin{proof}
	Assume that $f$ and $g$ have a common fixed point in $\partial \h$. By Lemma \ref{nocommonfix}, we may assume that $\langle f,g\rangle\subset\GL_2(\Z)\ltimes D_2$ and that $\langle f,g\rangle$ contains a subgroup $\Delta_2\subset D_2$ that is dense in $D_2$. We can write $f=d_1m_1$ and $g=d_2m_2$, where $d_1, d_2\in D_2$ and $m_1, m_2\in\GL_2(\Z)$. The group of diagonal automorphisms $D_2$ fixes the axes of all the monomial loxodromic elements (see \cite[Example 7.1]{cantatrev}). In particular, the loxodromic transformations $m_1$ and $m_2$ have the same fixed points on $\partial\h$ as $f$ and $g$ respectively. However, the group $\langle m_1, m_2\rangle$ does not contain any infinite abelian group, hence, by Lemma \ref{nocommonfix}, the transformations $m_1$ and $m_2$ do not have a common fixed point on $\partial\h$ - a contradiction.
\end{proof}

The main tool to prove the Tits alternative for subgroups of $\Cr_2(\C)$ containing loxodromic elements is the so-called ping-pong Lemma:

\begin{lemma}[Ping-pong Lemma, {\cite[II.B.]{MR1786869}}]
	Let $S$ be a set and $f_1, f_2$ two bijections of $S$. Assume that there exist subsets $S_1$ and $S_2$  of $S$ satisfying $S_1\not\subset S_2$  as well as $S_2\not\subset S_1$. If $f_1^n(S_1)\subset S_2$ and $f_2^n(S_2)\subset S_1$ for all $n\in\Z, n\neq 1$, then $f_1$ and $f_2$ generate a non-abelian free group.
\end{lemma}

\begin{lemma}\label{loxodromictits}
	Let $G\subset\Cr_2(\C)$ be a subgroup that contains a loxodromic element. Then one of the following is true:
	\begin{enumerate}
		\item $G$ is conjugate to a subgroup of $\GL_2(\Z)\ltimes D_2$;
		\item $G$ contains a subgroup $G^0$ of index at most two that is isomorphic to $\Z\ltimes H$, where $H$ is a finite group;
		\item $G$ contains a non-abelian free subgroup. 
	\end{enumerate}   
\end{lemma}

\begin{proof}
	Let $f\in G$ be a loxodromic element. We consider three cases.
	
	{\bf Case 1.} First we assume that all elements in $G$ preserve the axis $\Ax(f)$ of $f$. There is a subgroup $G^0\subset G$ of index at most $2$ such that $G^0$ preserves the orientation of the axis. Hence every element $g\in G^0$ translates the points on $\Ax(f)$ by a constant $c_g\in\R$. This yields a group homomorphism $\pi\colon G^0\to \R$, whose kernel is a bounded group. By Theorem \ref{cantatappendix}, $\ker(\pi)$ is either finite or $G$ is conjugate to a subgroup of $\GL_2(\Z)\ltimes D_2$. The image of $\pi$ is discrete by the gap property ( \cite[Corollary 2.7]{blanc2016dynamical}) and therefore isomorphic to $\Z$. Hence, if $\ker(\pi)$ is finite, we are in case (2).
	
	{\bf Case 2.} Now assume that there is an element $g\in G$ that does not preserve $\Ax(f)$. By Lemma \ref{nocommonfixgroup}, $G$ contains two loxodromic elements $h_1, h_2$ that do not have a common fixed point in $\partial\h$ and we apply the ping-pong Lemma by considering the action of $h_1$ and $h_2$ on the border $\partial\h$ and chosing as subsets $S_1$ and $S_2$ small enough neighborhoods of the fixed points of $f$ and $g$ on $\partial\h$. More precisely, denote by $\alpha^+$ the attracting fixed point of $f$ in $\partial\h$ and by $\alpha^-$ its repulsive fixed point. Similarly, we denote by $\beta^+$ and $\beta^-$ the attractive and repulsive fixed point of $g$ on $\partial\h$ respectively. Let $S^+_1$ be  a small neighborhood of $\alpha^+$ and $S^-_1$ a small neighborhood of $\alpha^-$ in $\partial\h$. Similarly, let  $S^+_2$ be  a small neighborhood of $\beta^+$ and $S^-_2$ a small neighborhood of $\beta^-$. We may assume that $S^+_1$, $S^-_1$, $S^+_2$, and $S^-_2$ are pairwise disjoint. Let $S_1\coloneqq S_1^+\cup S_1^-$ and $S_2\coloneqq S_2^+\cup S_2^-$. There exist positive integers $n_1, n_2, n_3, n_4$ satisfying $f^{n_1}(S_2)\subset S^+_1$, $f^{-n_2}(S_2)\subset S^-_1$, $g^{n_3}(S_1)\subset S_1^+$ and $g^{-n_3}(S_2)\subset S_1$. Define $n\coloneqq\max\{n_1, n_2, n_3, n_4\}$. As $f(S_1^+)\subset S_1^+$ and $f^{-1}(S_1^-)\subset S_1^-$ as well as $g(S_2^+)\subset S_2^+$ and $g^{-1}(S_2^-)\subset S_2^-$, we obtain that $f^m(S_2)\subset S_1$ and $f^{-m}(S_2)\subset S_1$ as well as  $g^m(S_1)\subset S_2$ and $g^{-m}(S_1)\subset S_2$ for all $m\geq n$. The two maps $f^n$ and $g^n$ together with the sets $S_1, S_2$ therefore satisfy all the conditions from the Ping-Pong Lemma and we obtain that $f^n$ and $g^n$ generate a non-abelian free subgroup of $G$.
\end{proof}

\begin{remark}
	We closely followed the proof from  \cite{MR2811600}, where  Cantat shows the Tits alternative for $\Cr_2(\C)$ for finitely generated groups. However, we use  a different argument to handle the case in which two loxodromic elements might have a common fixed point on $\partial\h$. The downside of the argument in \cite{MR2811600} is  that it relies on the ground field being of characteristic $0$. Whereas it looks like our approach to use Lemma \ref{nocommonfix} can be adapted to prove the Tits alternative for finitely generated subgroups of $\Cr_2(\kk)$, where $\kk$ is an algebraically closed field of arbitrary characteristic. 
\end{remark}

\begin{remark}\label{solvablegap}
	It seems that in the proof of the main theorem in \cite{MR3499516} it has not been considered that a priori there could be loxodromic elements with different axes but a common fixed point on $\partial\h$ (in this case the ping-pong Lemma can not be applied). However, Lemma \ref{nocommonfixgroup} fills this gap by showing that such loxodromic elements do not exist.
\end{remark}

\subsection{The parabolic case}
Recall that a subgroup of $\Cr_2(\C)$ that contains no loxodromic element, but a parabolic element always preserves a rational or elliptic fibration (Lemma \ref{noloxfibration}). From the structure theorems about these groups we will deduce that subgroups of this type satisfy the Tits alternative.

\begin{theorem}[\cite{cantat2012cremona}, Proposition 6.3]\label{titsextension}
	Assume that we have a short exact sequence of groups
	\[
	1\to G_1\to H\to G_2\to 1.
	\]
	If $G_1$ and $G_2$ satisfy the Tits alternative then $H$ satisfies the Tits alternative. 
\end{theorem}

\begin{remark}
	In the published version of the paper \cite{MR2811600} there is a gap in the proof of Theorem \ref{titsextension}. However, in the version of the paper on the website of the author \cite{cantat2012cremonawebsite} this gap has been filled and the proof is complete.
\end{remark}

\begin{lemma}\label{parabolictits}
	Let $G\subset\Cr_2(\C)$ be a subgroup that contains a parabolic element but no loxodromic element. Then $G$  satisfies the  Tits alternative.
\end{lemma}

\begin{proof}
	By Lemma \ref{noloxfibration}, $G$ is either conjugate to a subgroup of $\J$ or to a subgroup of $\aut(X)$, where $\aut(X)$ is the automorphism group of a Halphen surface. In the first case, the Tits alternative follows from Theorem \ref{titsextension} and the Tits alternative for linear groups in characteristic zero, since $J\simeq \PGL_2(\C)\ltimes\PGL_2(\C(t))$. In the second case, $G$ is solvable up to finite index since the automorphism group of a Halphen surface is virtually abelian by Theorem \ref{halphenstructure}.
\end{proof}

\subsection{Proof of Theorem \ref{titscremona}}
	The Tits alternative  for groups of birational transformations of non-rational complex compact K\"ahler surfaces has already been shown in \cite{MR2811600}. So it is enough to show it for $\Cr_2(\C)$. Let $G\subset\Cr_2(\C)$ be a subgroup. If $G$ contains a loxodromic element, then, by Lemma \ref{loxodromictits}, $G$ is either conjugate to a subgroup of $\GL_2(\Z)\ltimes D_2$, in which case the the Tits alternative holds by Theorem~\ref{titsextension}, or $G$ is cyclic up to finite index, or $G$ contains a non-abelian free subgroup. Therefore, the Tits alternative holds for groups containing loxodromic elements.
	
	For the case in which $G$ contains a parabolic but no loxodromic element, the Tits alternative is proven in Lemma \ref{parabolictits}.
	
	Assume that all elements in $G$ are elliptic. We thus are in one of the cases of Theorem \ref{main}. All the groups from case (1) or case (3) are isomorphic to a bounded group and hence in particular linear groups (see \cite[Remark 2.21]{MR3092478}). Therefore they satisfy the Tits alternative. If $G$ is a group from case (2)  the Tits alternative follows from Theorem \ref{titsextension}. 
\qed

\section{Solvable subgroups of $\Cr_2(\C)$}
In this section we prove Theorem \ref{derivedlength}. Our starting point is the following Theorem due to D\'eserti. Since the proof in \cite{MR3499516} seems to contain a small gap (see Remark \ref{solvablegap}), we will briefly recall the arguments and give a complete proof.  Moreover, Theorem \ref{main} allows us to refine the original result for the case of groups of elliptic elements:

\begin{theorem}\label{solvabledeserti}
	Let $G\subset\Cr_2(\C)$ be a solvable subgroup, then one of the following is true:
	\begin{enumerate}
		\item $G$ is a subgroup of elliptic elements, and hence isomorphic to a solvable subgroup of $\J\simeq\PGL_2(\C)\ltimes\PGL_2(\C(t))$ or to a solvable subgroup of a bounded group, i.e.\,a solvable subgroup from one of the groups in Theorem~\ref{maxsubgroups}.
		\item $G$ is conjugate to a subgroup of $\J\simeq\PGL_2(\C)\ltimes\PGL_2(\C(t))$.
		\item $G$ is conjugate to a subgroup of the automorphism group of a Halphen surface.
		\item $G$ is conjugate to a subgroup of $\GL_2(\Z)\ltimes D_2$.
		\item There is a loxodromic element $f\in\Cr_2(\C)$ and a finite group $H\subset\Cr_2(\C)$ such that $G=\Z\ltimes H$.
	\end{enumerate}
\end{theorem}

%
%

\begin{proof}
	Let $G\subset\Cr_2(\C)$ be a solvable subgroup. 
	
	{\bf Case 1.} $G$ contains a loxodromic element. In this case the statement follows directly from Lemma \ref{loxodromictits}.
	
	{\bf Case 2.} $G$ does not contain a loxodromic element, but $G$ contains a parabolic element. In this case $G$ is either a subgroup of the de Jonqui\`eres group $\J$, or $G$ is a subgroup of the automorphism group of a Halphen surface.
	
	{\bf Case 3.} $G$ is a group of elliptic elements and as such isomorphic to a subgroup of one of the groups from Theorem \ref{main}.
\end{proof}

We now use the classification of maximal algebraic subgroup (Theorem \ref{maxsubgroups}) to calculate the solvable length of bounded subgroups. Recall that if there is an exact sequence of groups
\[
1\to H_1\to G\to H_2\to 1.
\]
Then $G$ is solvable if and only if $H_1$ and $H_2$ are solvable. Moreover, the derived length of $G$ is at most the sum of the derived lengths of $H_1$ and of $H_2$. 

\begin{lemma}\label{boundedsolv}
The derived length of a bounded solvable subgroup of $\Cr_2(\C)$ is $\leq 5$.
\end{lemma}

\begin{proof}
	It is enough to consider solvable subgroups of the maximal algebraic subgroups of $\Cr_2(\C)$, i.e.\,the subgroups described in Theorem \ref{maxsubgroups}. Denote by $\psi(G)$ the maximal derived length of a solvable subgroup of a group $G$. In \cite{MR0302781} it is shown that $\psi(\GL_3(\C))=5$ and $\psi(\GL_2(\C))=4$. This implies that $\psi(\PGL_3(\C))\leq 5$ and $\psi(\PGL_2(\C))\leq 4$ and we obtain $\psi(\aut(\p^2))\leq 5, \psi(\aut(\p^1\times\p^1))=\psi((\PGL_2(\C)\times\PGL_2(\C))\rtimes\Z/2\Z)\leq 5$. Since the derived length of $\s_3$ is $2$, we obtain that $\psi(\aut(S_6))=3$. From Lemma \ref{GLmodmu} we deduce that $\psi(\aut(\F_n))\leq 5$ for all $n\geq 2$. Lemma \ref{semidirectexcept} implies that if $\pi\colon S\to\p^1$ is  an exceptional fibration, then $\psi(\aut(S,\pi))\leq 4$ and if $\pi\colon S\to\p^1$ is a $(\Z/2\Z)^2$-fibration, then $\psi(\aut(S,\pi))\leq 5$, by Theorem \ref{maxsubgroups}.
	
	So it remains to consider automorphism groups of del Pezzo surfaces of degree $\leq 5$, which are all finite. Let $S_5$ be the del Pezzo surface of degree $5$, then $\aut(S_5)=\s_5$ and hence $\psi(\aut(S_5)) \leq 3$. If $S_4$ is a del Pezzo surface of degree $4$, then $\aut(S_4)\simeq T\rtimes H_{\Delta}$, where $H_{\Delta}$ is isomorphic to $\Z/2\Z, \Z/4\Z, D_6$ or $D_{10}$ (see \cite[Section 3.4]{MR2504924}) and hence $\psi(\aut(S_4))\leq 3$. If $S_3$ is a del Pezzo surface of degree $3$ then $S_3$ is a cubic surface. A full list of all possible automorphism groups of cubic surfaces can be found in \cite[Table 4]{MR2641179}. One checks that in all the cases $\psi(\aut(S_3))\leq 3$. If $S_2$ is a del Pezzo surface of degree $2$ then $\aut(S_2)\simeq\Z/2\Z\times H_{S_\Delta}$, where $H_{S_\Delta}$ is a subgroup of $\aut(\p^2)$ (\cite[Section 3.6]{MR2504924}). We obtain $\psi(\aut(S_2))\leq 5$. The list of groups that can appear as automorphism groups of del Pezzo surfaces of degree $1$ can be found in \cite[Table 8]{MR2641179}. By checking all the groups that appear in the cited table, one obtains $\psi(\aut(S_1))\leq 4$ for all del Pezzo surfaces $S_1$ of degree $1$.
\end{proof}

\begin{proof}[Proof of Theorem \ref{derivedlength}]
	Let $G\subset\Cr_2(\C)$ be a solvable subgroup. 
	
	{\bf Case 1.} First we assume that $G$ contains a loxodromic element. By Theorem~\ref{solvabledeserti}, $G$ is either isomorphic to a subgroup of $\GL_2(\Z)\ltimes D_2$ or to a group of the form $\Z\ltimes H$, where $H\subset\Cr_2(\C)$ is finite and solvable. Every solvable subgroup of $\GL_2(\Z)$ has derived length $\leq 4$, hence in the first case, $G$ has derived length $\leq 5$. In the second case, by Lemma \ref{boundedsolv}, the derived length of $H$ is $\leq 5$ and hence the derived length of $G$ is $\leq 6$.
	
	{\bf Case 2.} In a next step we consider the case where $G$ does not contain a loxodromic element, but a parabolic element. In this case, $G$ is either isomorphic to a subgroup of the de Jonqui\`eres group $\J$ or to a subgroup of the automorphism group of a Halphen surface, by Lemma \ref{noloxfibration}. If $G$ is isomorphic to a subgroup of the de Jonqui\`eres group $\PGL_2(\C)\ltimes\PGL_2(\C(t))$  there exists a short exact sequence 
	\[
	1\to H_1\to G\to H_2\to 1,
	\]
	where $H_1\subset\PGL_2(\C(t))$ and $H_2\subset \PGL_2(\C)$. By  \cite{MR0302781}, the derived length of $H_1$ and $H_2$ is $\leq 4$ and hence the derived length of $G$ is $\leq 8$. If $G$ is isomorphic to a subgroup of the automorphism group of a Halphen surface $X$, then there exists, by Theorem \ref{halphenstructure}, a homomorphism $\rho\colon \aut(X)\to H$, where $H\subset\PGL_2(\C)$ is a finite solvable group and as such of derived length $\leq 4$. Moreover, $\ker(\rho)$ is an extension of an abelian group by a cyclic group of order dividing $24$. Hence the derived length of $G$ is $\leq 6$. 
	
	{\bf Case 3.} Finally, we consider the case where $G$ is a group of elliptic elements. By  Theorem \ref{main}, $G$ is either isomorphic to a subgroup of the de Jonquieres group or to a subgroup of a bounded group. In the first case we proceed as above, in the second case the claim is covered by Lemma \ref{boundedsolv}.
\end{proof}

\bibliographystyle{amsalpha}
\bibliography{/Users/christianurech/Dropbox/Literatur/bibliography_cu}

\end{document}